\documentclass[11pt,twoside]{amsart}
\usepackage{amssymb}
\catcode`@=11

\let\latex@newcommand\newcommand
\let\latex@renewcommand\renewcommand

\def\if@undefined#1{\ifx#1\un@@@defined@@@}
\def\newcommand#1{\if@undefined{#1}
  \let\next@=\latex@newcommand \else
  \let\next@=\latex@renewcommand \fi
  \next@{#1}}

\catcode`@=12

\begin{document}
\theoremstyle{plain}
\newtheorem{thm}{Theorem}[section]
\newtheorem*{thm*}{Theorem}
\newtheorem{prop}[thm]{Proposition}
\newtheorem*{prop*}{Proposition}
\newtheorem{lemma}[thm]{Lemma}
\newtheorem{cor}[thm]{Corollary}
\newtheorem*{conj*}{Conjecture}
\newtheorem*{cor*}{Corollary}
\newtheorem{defn}[thm]{Definition}
\newtheorem{cond}{Condition}
\theoremstyle{definition}
\newtheorem*{defn*}{Definition}
\newtheorem{rems}[thm]{Remarks}
\newtheorem*{rems*}{Remarks}
\newtheorem*{proof*}{Proof}
\newtheorem*{not*}{Notation}
\newcommand{\npartial}{\slash\!\!\!\partial}
\newcommand{\Heis}{\operatorname{Heis}}
\newcommand{\Solv}{\operatorname{Solv}}
\newcommand{\Spin}{\operatorname{Spin}}
\newcommand{\SO}{\operatorname{SO}}
\newcommand{\ind}{\operatorname{ind}}
\newcommand{\Index}{\operatorname{index}}
\newcommand{\ch}{\operatorname{ch}}
\newcommand{\rank}{\operatorname{rank}}
\newcommand{\abs}[1]{\lvert#1\rvert}
 \newcommand{\A}{{\mathcal A}}
 \newcommand{\E}{{\mathcal E}}
        \newcommand{\D}{{\mathcal D}}\newcommand{\HH}{{\mathcal H}}
        \newcommand{\LL}{{\mathcal L}}
        \newcommand{\B}{{\mathcal B}}
         \newcommand{\S}{{\mathcal S}}
 \newcommand{\F}{{\mathcal F}}

        \newcommand{\K}{{\mathcal K}}
\newcommand{\oo}{{\mathcal O}}
         \newcommand{\PP}{{\mathcal P}}
        \newcommand{\s}{\sigma}
\newcommand{\al}{\alpha}
        \newcommand{\coker}{{\mbox coker}}
        \newcommand{\p}{\partial}
        \newcommand{\dd}{|\D|}
        \newcommand{\n}{\parallel}
\newcommand{\bma}{\left(\begin{array}{cc}}
\newcommand{\ema}{\end{array}\right)}
\newcommand{\bca}{\left(\begin{array}{c}}
\newcommand{\eca}{\end{array}\right)}
\def\clsp{\overline{\operatorname{span}}}
\def\T{\mathbb T}
\def\Aut{\operatorname{Aut}}

\newcommand{\sr}{\stackrel}
\newcommand{\da}{\downarrow}
\newcommand{\tD}{\tilde{\D}}

        \newcommand{\R}{\mathbf R}
        \newcommand{\C}{\mathbf C}
        \newcommand{\h}{\mathbf H}
\newcommand{\Z}{\mathbf Z}
\newcommand{\N}{\mathbf N}
\newcommand{\tto}{\longrightarrow}
\newcommand{\ben}{\begin{displaymath}}
        \newcommand{\een}{\end{displaymath}}
\newcommand{\be}{\begin{equation}}
\newcommand{\ee}{\end{equation}}

        \newcommand{\bean}{\begin{eqnarray*}}
        \newcommand{\eean}{\end{eqnarray*}}
\newcommand{\nno}{\nonumber\\}
\newcommand{\bea}{\begin{eqnarray}}
        \newcommand{\eea}{\end{eqnarray}}

\def\cross#1{\rlap{\hskip#1pt\hbox{$-$}}}
        \def\intcross{\cross{0.3}\int}
        \def\bigintcross{\cross{2.3}\int}

\newcommand{\supp}[1]{\operatorname{#1}}
\newcommand{\norm}[1]{\parallel\, #1\, \parallel}
\newcommand{\ip}[2]{\langle #1,#2\rangle}
\setlength{\parskip}{.3cm}
\newcommand{\nc}{\newcommand}
\nc{\nt}{\newtheorem} \nc{\gf}[2]{\genfrac{}{}{0pt}{}{#1}{#2}}
\nc{\mb}[1]{{\mbox{$ #1 $}}} \nc{\real}{{\mathbb R}}
\nc{\comp}{{\mathbb C}} \nc{\ints}{{\mathbb Z}}
\nc{\Ltoo}{\mb{L^2({\mathbf H})}} \nc{\rtoo}{\mb{{\mathbf R}^2}}
\nc{\slr}{{\mathbf {SL}}(2,\real)} \nc{\slz}{{\mathbf
{SL}}(2,\ints)} \nc{\su}{{\mathbf {SU}}(1,1)} \nc{\so}{{\mathbf
{SO}}} \nc{\hyp}{{\mathbb H}} \nc{\disc}{{\mathbf D}}
\nc{\torus}{{\mathbb T}}
\newcommand{\tk}{\widetilde{K}}
\newcommand{\boe}{{\bf e}}\newcommand{\bt}{{\bf t}}
\newcommand{\vth}{\vartheta}
\newcommand{\CGh}{\widetilde{\CG}}
\newcommand{\db}{\overline{\partial}}
\newcommand{\tE}{\widetilde{E}}
\newcommand{\tr}{\mbox{tr}}
\newcommand{\ta}{\widetilde{\alpha}}
\newcommand{\tb}{\widetilde{\beta}}
\newcommand{\txi}{\widetilde{\xi}}
\newcommand{\hV}{\hat{V}}
\newcommand{\IC}{\mathbf{C}}
\newcommand{\IZ}{\mathbf{Z}}
\newcommand{\IP}{\mathbf{P}}
\newcommand{\IR}{\mathbf{R}}
\newcommand{\IH}{\mathbf{H}}
\newcommand{\IG}{\mathbf{G}}
\newcommand{\CC}{{\mathcal C}}
\newcommand{\CS}{{\mathcal S}}
\newcommand{\CG}{{\mathcal G}}
\newcommand{\CL}{{\mathcal L}}
\newcommand{\CO}{{\mathcal O}}
\nc{\ca}{{\mathcal A}} \nc{\cag}{{{\mathcal A}^\Gamma}}
\nc{\cg}{{\mathcal G}} \nc{\chh}{{\mathcal H}} \nc{\ck}{{\mathcal
B}} \nc{\cl}{{\mathcal L}} \nc{\cm}{{\mathcal M}}
\nc{\cn}{{\mathcal N}} \nc{\cs}{{\mathcal S}} \nc{\cz}{{\mathcal
Z}} \nc{\cM}{{\mathcal M}}
\nc{\sind}{\sigma{\rm -ind}}
\newcommand{\la}{\langle}
\newcommand{\ra}{\rangle}

\renewcommand{\labelitemi}{{}}

\def\title#1{\begin{center}\bf\large #1\end{center}\vskip 0.3 in}
\def\author#1{\begin{center}#1\end{center}}

\title{TWISTED CYCLIC THEORY AND AN INDEX THEORY FOR THE GAUGE INVARIANT KMS STATE ON CUNTZ ALGEBRAS}


\centerline {A.L. Carey$^\ast$, J. Phillips$^\sharp$, A. Rennie$^{\ast,\dag}$}
\vspace*{0.1in}

\centerline {$^\ast$ Mathematical Sciences Institute, Australian National University,
Canberra, ACT, AUSTRALIA}

\centerline{ $^\sharp$ Department of Mathematics and Statistics,
University of Victoria,
Victoria, BC,  CANADA}

\centerline{$^\dag$ Department of Mathematics, University of Copenhagen, Copenhagen, Denmark}

\email{ acarey@maths.anu.edu.au, phillips@math.uvic.ca,
        rennie@math.ku.dk}


\vspace*{0.1in}

\centerline{\bf Abstract}
This paper presents, by example, an index theory
appropriate to algebras without trace. Whilst we work exclusively with
the Cuntz algebras the exposition is designed to indicate how to develop a general theory. 
Our main result is an index theorem (formulated in terms of
spectral flow) using a twisted cyclic cocycle where
the twisting comes from the modular automorphism group
for the canonical gauge action on the Cuntz algebra. 
We introduce a modified $K_1$-group of the Cuntz algebra 
so as to pair with this twisted cocycle. As a corollary we
obtain a
noncommutative geometry interpretation for Araki's notion of relative entropy 
in
this example. We also note the connection of this example to the theory of noncommutative manifolds.

\setlength{\parskip}{.01cm}
\tableofcontents
\setlength{\parskip}{.3cm}
\section{Introduction}

In this paper we initiate
an  extension of index theory
to algebras without trace. We take the Cuntz algebras $O_n$ \cite{Cu}
as  basic
examples.
 In the absence of a non-trivial trace on the Cuntz algebras, 
 our approach is to use a
KMS state, \cite{BR2},
to define an index pairing using spectral flow.
The state we use is the unique KMS state for the canonical
$\T^1$ gauge action on $O_n$. As $O_n$
is a graph algebra, we can import many of the techniques of
\cite{pr} where the semifinite version of the local index formula was used 
to calculate spectral flow invariants of a class of Cuntz-Krieger algebras. 
The Cuntz algebras give us an excellent testing ground for the ideas 
required to deal
with index theory in a type III setting.

The approach is motivated by
\cite{CPRS2} where a semifinite local index formula 
in noncommutative geometry is  proved. 
This semifinite theory is reviewed 
 in Section \ref{background} together with  notation for the Cuntz algebras.
In \cite{pr} the semifinite  theory was applied to 
certain graph $C^*$-algebras.
The new idea explained there was
the construction of a Kasparov $A,F$-module 
for the graph algebra $A$ of a locally
finite graph with no sources where $F=A^{\T^1}$ is the  
fixed point algebra  for the natural 
$\T^1$ gauge action. This construction applies to the Cuntz algebra because it is a graph algebra of this type.
A $K$-theoretic refinement for the
index theorem in \cite{pr} was developed
in \cite{CPR} where the odd Kasparov module of \cite{pr}
is `doubled up' on a half
infinite cylinder to an even  Kasparov $M(F,A),F$-module, where $M(F,A)$ is
the mapping cone algebra for the inclusion of the fixed point algebra.
Our idea is  to modify this tracial case so as to extend, 
as far as is possible, these results to
the Cuntz algebra.

We easily observe that there is a Kasparov module for the Cuntz algebra and hence that we
have a $K_0(F)$-valued pairing with $M(F,A)$.
However, in the absence of a trace we need a new idea.
{\bf The primary result of this
paper introduces a modified
spectral triple
(referred to as a `modular spectral triple') with which we can
compute an index pairing.}  
Our method, of employing a 
KMS functional instead of a trace, leads to various subtleties. Restricting the KMS state to the fixed point algebra $F$ gives a trace on $F$, and so a homomorphism on $K_0(F)$. However, when we pass to the Morita equivalent algebra of compact endomorphisms on our Kasparov module, we find that the functional we are forced to employ on this new algebra does not respect all Murray-von Neumann equivalences.
It is this fact that leads to the consideration of  finer invariants
than those obtained from ordinary K-theory in the KMS or `twisted setting'.

{\bf We show that modular spectral triples lead to
`twisted residue cocycles'
using a variation on the semifinite residue cocycle
of \cite{CPRS2}}. 
It is well known that such twisted cocycles cannot pair with ordinary $K_1$, rather
we introduce, in Section 4, a substitute which we term `modular $K_1$'.
It is a semigroup 
and, as is 
explained in 
our main theorem (Theorem 5.5), there is
a general spectral flow formula which defines
the pairing of modular $K_1$
with our `twisted residue cocycle'.
There is an analogy with the local index formula
of noncommutative geometry 
in the $\LL^{1,\infty}$-summable
 case, however, there are important
differences: the usual residue cocycle is replaced
by a twisted residue cocycle and the Dixmier trace 
arising in the standard
situation is replaced by a KMS-Dixmier functional. 
The common ground with \cite{CPRS2} stems from the
use of the general spectral flow formula of \cite{CP2} to derive
the twisted residue cocycle and this has the corollary that we have a homotopy invariant.

For the Cuntz algebras the main result is
 Theorem 5.6 and its Corollary where we compute, for particular modular unitaries
in matrix algebras over the Cuntz algebras, the precise numerical 
values arising from the general formalism. We use \cite{CPR}
to see that these numerical values
provide strong evidence that the mapping cone KK-theory of Section 2
is playing a (yet to be fully understood) role.
In the final Section we note that there is a physical interpretation of the
spectral flow invariant we are calculating in terms of Araki's notion of
relative entropy of two KMS states. We also show that our modular spectral triples for the Cuntz algebras satisfy twisted versions of Connes' axioms for noncommutative manifolds.

We plan to return to this matter and to the appropriate cohomological
setting for our index theorem elsewhere. 
Already, in work in progress  \cite{crt},
  we have uncovered further
examples which indicate there is a complex and interesting theory to be
understood.

The organisation is summarised in the Contents list.
Section 2 is review material which places this article in context. 
The Cuntz algebra example begins on Section 3
and the main new material is in
 Sections 4 and 5.

{\bf Acknowledgements} We would like to thank Ian Putnam,
Nigel Higson, Ryszard Nest, Sergey Neshveyev and 
Kester Tong for advice and comments. The first named author
was supported by the Australian Research
Council, the Clay Mathematics Institute,
and  the Erwin Schrodinger Institute  (where some of this
paper was written). The second named author acknowledges the support of
NSERC (Canada) while the third named author thanks  Statens
Naturvidenskabelige Forskningsr{\aa}d, Denmark. All authors are grateful for 
the support of the Banff International Research Station where some of this 
research was undertaken.

\section{Some background} 
\label{background}
\subsection{Semifinite noncommutative geometry}
We begin with some semifinite versions of standard definitions and
results following \cite{CPRS2}. 
Let $\phi$ be a fixed faithful, normal, semifinite trace
on a von Neumann algebra ${\mathcal N}$. Let ${\mathcal
K}_{\mathcal N }$ be the $\phi$-compact operators in ${\mathcal
N}$ (that is the norm closed ideal generated by the projections
$E\in\mathcal N$ with $\phi(E)<\infty$).

\begin{defn} A {\bf semifinite
spectral triple} $(\A,\HH,\D)$ is given by a Hilbert space $\HH$, a
$*$-algebra $\A\subset \cn$ where $\cn$ is a semifinite von
Neumann algebra acting on $\HH$, and a densely defined unbounded
self-adjoint operator $\D$ affiliated to $\cn$ such that
$[\D,a]$ is densely defined and extends to a bounded operator
in $\cn$ for all $a\in\A$ and $(\lambda-\D)^{-1}\in\K_\cn$ for all $\lambda\not\in{\R}.$
The triple is said to be {\bf even} if there is $\Gamma\in\cn$ such
that $\Gamma^*=\Gamma$, $\Gamma^2=1$,  $a\Gamma=\Gamma a$ for all
$a\in\A$ and $\D\Gamma+\Gamma\D=0$. Otherwise it is {\bf odd}.
\end{defn}



Note that if $T\in\cn$ and
$[\D,T]$ is bounded, then $[\D,T]\in\cn$. 

We recall from \cite{FK}
that if $S\in\mathcal N$, the {\bf t-th generalized singular
value} of $S$ for each real $t>0$ is given by
$$\mu_t(S)=\inf\{||SE||\ : \ E \mbox{ is a projection in }
{\mathcal N} \mbox { with } \phi(1-E)\leq t\}.$$
The ideal $\LL^1({\mathcal N},\phi)$ consists of those operators $T\in
{\mathcal N}$ such that $\n T\n_1:=\phi( |T|)<\infty$ where
$|T|=\sqrt{T^*T}$. In the Type I setting this is the usual trace
class ideal. We will denote the norm on $\LL^1(\cn,\phi)$ by
$\n\cdot\n_1$. An alternative definition in terms of singular
values is that $T\in\LL^1(\cn,\phi)$ if $\|T\|_1:=\int_0^\infty \mu_t(T) dt
<\infty.$
When ${\mathcal N}\neq{\mathcal
B}({\mathcal H})$, $\LL^1(\cn,\phi)$ need not be complete in this norm but it is
complete in the norm $||.||_1 + ||.||_\infty$. (where
$||.||_\infty$ is the uniform norm). We use the notation
$${\mathcal L}^{(1,\infty)}({\mathcal N},\phi)=
\left\{T\in{\mathcal N}\ : \Vert T\Vert_{_{{\mathcal
L}^{(1,\infty)}}} :=   \sup_{t> 0}
\frac{1}{\log(1+t)}\int_0^t\mu_s(T)ds<\infty\right\}.$$


 The reader
should note that ${\mathcal L}^{(1,\infty)}(\cn,\phi)$ is often taken to
mean an ideal in the algebra $\widetilde{\mathcal N}$ of
$\phi$-measurable operators affiliated to ${\mathcal N}$. Our
notation is however consistent with that of \cite{C} in the
special case ${\mathcal N}={\mathcal B}({\mathcal H})$. With this
convention the ideal of $\phi$-compact operators, ${\mathcal
  K}({\mathcal N})$,
consists of those $T\in{\mathcal N}$ (as opposed to
$\widetilde{\mathcal N}$) such that $\mu_\infty(T):=\lim
_{t\to \infty}\mu_t(T)  = 0.$

\begin{defn}\label{summable} A semifinite  spectral triple
$(\A,\HH,\D)$ relative to $(\cn,\phi)$
with $\A$ unital is
$(1,\infty)$-summable if 
$(\D-\lambda)^{-1}\in\LL^{(1,\infty)}(\cn,\phi)\ \mbox{for all}\ 
\lambda\in\C\setminus\R.$ 
\end{defn}

It follows that if $(\A,\HH,\D)$ is $(1,\infty)$-summable then it is
$n$-summable (with respect to the trace $\phi$) for all $n>1$.   
We next need to briefly discuss Dixmier traces. For
more information on semifinite Dixmier traces, see \cite{CPS2}.
For $T\in\LL^{(1,\infty)}(\cn,\phi)$, $T\geq 0$, the function \ben
F_T:t\to\frac{1}{\log(1+t)}\int_0^t\mu_s(T)ds \een is bounded. 
There are certain $\omega\in L^\infty(\R_*^+)^*$, \cite{CPS2,C}, which
define (Dixmier) traces on
 $\LL^{(1,\infty)}(\cn,\phi)$ by setting
$$ \phi_\omega(T)=\omega(F_T), \ \ T\geq 0$$
and  extending to all of
$\LL^{(1,\infty)}(\cn,\phi)$ by linearity.
For each such $\omega$ we write $\phi_\omega$ for
the associated Dixmier trace.
Each Dixmier trace $\phi_\omega$ vanishes on the ideal of trace class
operators. Whenever the function $F_T$ has a limit at infinity,
all Dixmier traces return that limit as their value. 
This leads to the notion of a measurable operator \cite{C,LSS},
that is, one on which all Dixmier traces take the same value.

We now introduce (a special case of) 
the analytic spectral flow formula of 
\cite{CP1,CP2}. 
This formula starts with a semifinite spectral triple $(\A,\HH,\D)$
and computes the $\phi$ spectral flow from $\D$ to $u\D u^*$, 
where $u\in\A$ is unitary with $[\D,u]$ bounded, in the case
where $(\A,\HH,\D)$ is $n$-summable for $n> 1$ (Theorem 9.3 of \cite{CP2}):
\be sf_\phi(\D,u\D u^*)
=\frac{1}{C_{n/2}}\int_0^1\phi(u[\D,u^*](1+(\D+tu[\D,u^*])^2)^{-n/2})dt,
\label{basicformula}\ee
with $C_{n/2}=\int_{-\infty}^\infty(1+x^2)^{-n/2}dx$. 
This real number $sf_\phi(\D,u\D u^*)$
is a pairing of the $K$-homology class
$[\D]$ of $\mathcal A$ with the $K_1({\mathcal A})$ class $[u]$
\cite{CPRS2}. There is a geometric way to view this formula. It is shown in 
\cite {CP2} that
the functional $X\mapsto \phi(X(1+(\D+Y)^2)^{-n/2})$ 
determines an exact one-form for $X$ in the tangent space, ${\mathcal N}_{sa},$ of
an affine space $\D+{\mathcal N}_{sa}$ modelled on  ${\mathcal N}_{sa}$. 
Thus (\ref{basicformula})
represents the integral of this one-form along the path 
$\{\D_t=(1-t)\D+ tu\D u^*\}$
provided one appreciates that $\dot\D_t=u[\D,u^*]$ is a tangent vector
to this path.
In \cite{CPRS2}, the local index formula in noncommutative geometry 
of \cite{CM} was extended to semifinite spectral triples. In the
simplest terms, the local index formula is a
pairing of a finitely summable spectral triple $(\A,\HH,\D)$ with
the $K$-theory of the $C^*$-algebra $\overline{\A}$. Our approach in this paper
is inspired by the following theorem (see also
\cite{CPRS2,CM,Hig}).
\begin{thm}[\cite{CPS2}] Let $(\A,\HH,\D)$ be an odd 
$(1,\infty)$-summable semifinite spectral triple, relative
to $(\cn,\phi)$. Then for $u\in\A$ unitary the pairing of $[u]\in
K_1(\overline{\A})$ with $(\A,\HH,\D)$ is given by
$$ \la
[u],(\A,\HH,\D)\ra=sf_\phi(\D,u\D u^*)=
\lim_{s\to 0^+}s\ \phi(u[\D,u^*](1+\D^2)^{-1/2-s}).$$ 
In particular, the
limit on the right exists.
\end{thm}

\subsection{The Cuntz algebras and the canonical Kasparov module}\label{graphalg}

For $n\geq 2$, the Cuntz algebra \cite{Cu} on $n$ generators, $O_n$, is the 
(universal) $C^*$-algebra generated by
$n$ isometries $S_i$, $i=1,...,n$, subject only to the relation
$\sum_{i=1}^n S_iS_i^*=1.$ The projections $S_iS_i^*$ will
be denoted by $P_i$ and more generally we will write $P_\mu=S_\mu
S_\mu^*$.  For $\mu\in \{1,2,...,n\}^k={\bf n}^k$ we write 
$S_\mu=S_{\mu_1}S_{\mu_2}\cdots S_{\mu_k}$  and 
$S^*_{\mu}=S^*_{\mu_k}S^*_{\mu_{k-1}}\cdots S^*_{\mu_1}.$ Using
the fact that $S^*_iS_j=\delta_{ij}$, one can show that every word
in the $S_i,S^*_j$ can be written in the form $S_\mu S^*_\nu$,
where $\mu\in{\bf n}^k$ and $\nu\in{\bf n}^l$ are  multi-indices.
We will write $|\mu|=k$ and $|\nu|=l$ for the length of such
multi-indices.
As the family of monomials $\{S_\mu S_\nu^*\}$
is closed under multiplication and involution, we have
\begin{equation}
O_n=\clsp\{S_\mu S_\nu^*:\mu\in{\bf n}^k,\ \nu\in{\bf n}^m,\
k,m\geq 0\}.\label{cuntzspanningset}
\end{equation}
If $z\in \T^1$, then the family $\{zS_j\}$ is another
Cuntz-Krieger family which generates $O_n$, 
and the universal
property of $O_n$ gives a homomorphism
$\sigma_z:O_n\to O_n$ such that $\sigma_z(S_e)=zS_e$. The homomorphism $\sigma_{\overline z}$ is an
inverse for $\sigma_z$, so $\sigma_z\in\Aut O_n$, and a routine
argument shows that
$\sigma$ is a strongly continuous action of $\T^1$ on $O_n$. It is
called the \emph{gauge action}. Averaging over $\sigma$ with respect to normalised Haar measure
gives a positive, faithful expectation $\Phi$ of $O_n$ onto the fixed-point algebra
$F:=O_n^\sigma$:
\[
\Phi(a):=\frac{1}{2\pi}\int_{\T^1} \sigma_z(a)\,d\theta\ \mbox{ for
}\ a\in O_n,\ \ z=e^{i\theta}.
\]



To simplify notation, we let $A=O_n$ be the Cuntz algebra and
$F=A^\sigma$, the fixed point algebra for the $\T^1$ gauge
action. 
The algebras $A_c,F_c$ are defined as the finite linear
span of the generators.
Right multiplication makes $A$ into a right $F$-module, and
similarly $A_c$ is a right module over $F_c$. We define an
$F$-valued inner product $(\cdot|\cdot)_R$ on both these modules
by
 $ (a|b)_R:=\Phi(a^*b).$
\begin{defn}\label{mod} Let $X$ be the right $F$ $C^*$-module obtained by
completing $A$ (or $A_c$) in the norm
$$ \Vert x\Vert^2_X:=\Vert (x|x)_R\Vert_F=\Vert
\Phi(x^*x)\Vert_F.$$
\end{defn}
The algebra $A$ acting by left multiplication on  $X$
provides a representation of $A$ as adjointable operators on $X$.
Let $X_c$ be the copy of $A_c\subset X$. The $\T^1$ action on $X_c$ is
unitary and extends to $X$, \cite{pr}.
For all $k\in\Z$, the  projection onto the $k$-th spectral subspace
of the $\T^1$ action is the operator $\Phi_k$ on $X$:  
$$\Phi_k(x)=\frac{1}{2\pi}\int_{\T^1}z^{-k}\sigma_z(x)d\theta,
\ \ z=e^{i\theta},\ \ x\in X.$$ Observe
 that $\Phi_0$ restricts to $\Phi$ on $A$
and on generators of $O_n$ we have \begin{equation}\Phi_k(S_\al
S_\beta^*)=\left\{\begin{array}{lr}S_\al S_\beta^* & \ \
|\al|-|\beta|=k\\0 & \ \ |\al|-|\beta|\neq
k\end{array}\right..\label{kthproj}\end{equation}

We quote the following result from \cite{pr}.
\begin{lemma}\label{phiendo} The operators $\Phi_k$ are adjointable 
endomorphisms of the $F$-module $X$ such that $\Phi_k^*=\Phi_k=\Phi_k^2$ and
$\Phi_k\Phi_l=\delta_{k,l}\Phi_k$. If $K\subset\Z$ then the sum
$\sum_{k\in K}\Phi_k$ converges strictly to a projection in the
endomorphism algebra. The sum $\sum_{k\in\Z}\Phi_k$ converges to
the identity operator on $X$. For all $x\in X$, the sum
$x=\sum_{k\in\Z}\Phi_kx=\sum_{k\in\Z}x_k$ converges in $X$.
\end{lemma}
The unbounded operator of the next proposition is of course
the generator of the
$\T^1$ action on $X$.
We refer to Lance's book, \cite[Chapters 9,10]{L}, for information
on unbounded operators on $C^*$-modules.

\begin{prop}\label{dee}\cite{pr} Let $X$ be the right $C^*$-$F$-module of
Definition \ref{mod}. Define $\D:X_\D\subset X$ to be the linear space
$$ X_\D=
\{x=\sum_{k\in\Z}x_k\in X:\Vert \sum_{k\in\Z}k^2(x_k|x_k)_R\Vert<\infty\}.$$
For $x\in X_\D$ define $ \D(x)=\sum_{k\in\Z}kx_k.$  Then $\D:X_\D\to X$ is a
is self-adjoint, regular operator on $X$.
\end{prop}

{\bf Remark}. On generators in $O_n$ (regarded as
elements of $X_c\subset X$) we have 
$\D(S_\al S_\beta^*)=(|\al|-|\beta|)S_\al S_\beta^*.$

We will need the following technical result from \cite{pr} later:

\begin{lemma}\label{finrank}  For all $a\in A$ and $k\in\Z$,
$a\Phi_k\in End^0_F(X)$, the
 compact $F$ linear 
endomorphisms of the right $F$ module $X$. If $a\in A_c$ then
$a\Phi_k$ is finite rank.
\end{lemma}

Introduce the rank one operator
 $\Theta^{R}_{x,y}$ by $\Theta^{R}_{x,y}z=x(y|z)_R$. Then by \cite[Lemma 4.7]{pr}, 
for $k\geq 0$, $\Phi_k=\sum_{|\mu|=k}\Theta^{R}_{S_\mu,S_\mu}$
where for the Cuntz algebras the sum is finite.
For the negative subspaces the formula
in \cite{pr} gives, in the Cuntz algebras 
$\Phi_{-k}=\frac{1}{n^k}\sum_{|\mu|=k}\Theta^{R}_{S_\mu^*,S_\mu^*}
.$

\begin{thm} \cite{pr} Let $X$ be the right $F$ module of Definition \ref{mod}. Let
$V=\D(1+\D^2)^{-1/2}$. Then $(X,V)$ is an odd Kasparov module for
$A$-$F$ and so defines an element of $KK^1(A,F)$.
\end{thm}
Given the hypotheses of the Theorem, we may write $\D$ as
$\D=\sum_{k\in\Z}k\Phi_k.$

{\bf Remarks}. The constructions in \cite{pr} imply immediately
that  we obtain a class in
$KK^1(O_n,F)$.
Theorem 2.8 is part of an index theorem proved in
\cite{pr}. The pairing of $(X,V)$ with unitaries $u$ in $K_1(A)$ 
gives a $K_0(F)$
valued index, and writing $P=\mathcal{X}_{[0,\infty)}(\D)$, it is given by
\begin{equation}\label{pairing}
\la [u],[(X,V)]\ra=[\ker(PuP)]-[\mbox{coker}(PuP)]
\end{equation}
where the square brackets denote the $K_0$ class of the relevant
kernel projections. 
However, the main result of \cite{pr} (which fails for the Cuntz algebras)
requires a faithful semifinite gauge invariant lower
semi-continuous trace $\phi$ on $A$. 

\subsection{The mapping cone algebra and APS boundary conditions}

In \cite{CPR} we refined \cite{pr}
by showing that $K_0(F)$-valued  
indices could also be obtained  
from an even index pairing in $KK$-theory
using APS boundary conditions,
similar to \cite{APS3}. We briefly review this result, as it provides
an interpretation of the modular index pairings in Section 5.
We use the notation $M_k(B)$ to denote the algebra of
$k\times k$ matrices over an algebra $B$.
If $F\subset A$ is a sub-$C^*$-algebra of the $C^*$-algebra $A$, then
the mapping cone algebra for the inclusion is 
$$M(F,A)=\{f:\R_+=[0,\infty)\to A: f\ \mbox{is continuous and vanishes at
infinity},\ f(0)\in F\}.$$
When $F$ is an ideal in $A$ it is known that $K_0(M(F,A))\cong
K_0(A/F)$, \cite{Put}.  In general, $K_0(M(F,A))$ is the set of
homotopy classes of partial isometries $v\in M_k(A)$ with range and
source projections $vv^*,\ v^*v$ in $M_k(F)$, with operation the
direct sum and inverse $-[v]=[v^*]$. All this is proved in \cite{Put}.

Following \cite{CPR} we now explain our noncommutative analogue
of the Atiyah-Patodi-Singer index theorem \cite{APS1}.
Note that when we are working with matrix algebras over $A$ or $M(F,A)$
we inflate $\D$ to $\D\otimes I_k$ and so on.
\begin{defn} Let $(X,\D)$ be an unbounded  Kasparov module
and form the algebraic tensor product of  $\LL^2(\R_+)$ and $X$.
We complete the linear span of the elementary tensors in the algebraic 
tensor product (these are functions
from $\R_+$ to $X$) in the norm arising from the inner product
$$\la\xi,\eta\ra:=\int_0^\infty(\xi(t)|\eta(t))_Xdt$$
and write the completion as  $\LL^2(\R_+)\otimes X$ and denote this space by
$\E.$
An {\bf extended $\LL^2$-function} $f:\R_+\to X$ is a
function of the form $f=g+x_0$ such that
$g$ is in $\LL^2(\R_+)\otimes X$ and $x_0$ is a constant function with
$\D(x_0)=0$, that is $x_0\in{\rm ker}(\D)=\Phi_0(X).$ We denote the space of 
extended $\LL^2$-functions by $\hat{\E}$ and define the $F$-valued inner product 
on $\hat{\E}$ by $\la g+x|h+y\ra_{\hat{\E}} := \la g|h\ra_{\E} +\la x|y\ra_X.$
\end{defn} 

Now, certain Kasparov
$A,F$-modules extend to Kasparov $M(F,A),F$-modules:

\begin{prop}[\cite{CPR}]\label{thmone} Let $(X,\D)$ be an ungraded 
unbounded Kasparov module
for $C^*$-algebras $A,F$ with $F\subset A$ a subalgebra
such that $\overline{AF}=A$. Suppose that
$\D$ also commutes with the left action of $F\subset A$, and that $\D$
has 
discrete spectrum. Then the pair
$$(\hat{X},\hat\D)=\left(\bca \E\\ \hat{\E} \eca,
\bma 0 & -\p_t+\D\\
\p_t+\D & 0\ema\right)$$
with APS boundary conditions is a graded unbounded Kasparov module for
the mapping cone algebra $M(F,A)$.
\end{prop}

 By APS boundary conditions we mean let $P=\mathcal {X}_{\R_+}(\D)$
and take the domain
of $\hat\D$ to initially be   
$$\mbox{dom}\hat\D=\{\xi \in \mbox{span of elementary tensors in } 
\hat{X}:\ P\xi_1(0)=0,\ (1-P)\xi_2(0)=0,\
\hat\D\xi\in\hat{X}\}.$$ 
In \cite{CPR}
we show that APS boundary conditions make sense for the self adjoint
closure
of $\hat\D$ and no technical obstructions exist to working with  this
closure on its natural domain. Strictly speaking we should also mention
that the unbounded Kasparov
module is defined for a certain smooth algebra $\A\subset A$, and we
will suppose that this is the case, and that $F\subset\A$.
To explain the appearance in the second component of $\hat{X}$
of the right $F$, $C^*$-module 
$\hat{\E}$, we have to recall 
that the  different treatment of ${\rm ker}(\D)$ (which has the restriction 
of the inner product on $X$) is to account for
`extended $\LL^2$ solutions' corresponding to the zero eigenvalue of
$\D$, just as in \cite[pp 58-60]{APS1}.

If $v$ is a partial isometry in  $M_k(\A)^\sim$ (the minimal unitization)
setting $$e_v(t)=\bma
1-\frac{vv^*}{1+t^2} & -iv\frac{t}{1+t^2}\\iv^*\frac{t}{1+t^2}&
\frac{v^*v}{1+t^2}\ema,$$
defines $e_v$ as a projection
in $M_k(F,A)$ (the $k\times k$ matrices over the mapping cone). 
When $v$ is a unitary, denoted $u$ say, 
then $u$ is trivially a partial isometry with
range and source in (the unitization of) 
$M_k(F)$, so we obtain a class in $K_0(M(F,A))$ which
we denote by $[e_u]-\left[\bma 1 & 0\\ 0 & 0\ema\right]$.
In the statement of the next result 
(which is a special case of the main theorem of \cite{CPR})
we suppress the subscript $k$.

\begin{prop}[\cite{CPR}]\label{mainresult} 
 Let $(X,\D)$ be an ungraded unbounded Kasparov module
for (pre-)$C^*$-algebras $\A,F$ with $F\subset \A$ a subalgebra
such that $\overline{AF}=F$. Suppose that
$\D$ also commutes with the left action of $F\subset \A$, and that $\D$
has 
discrete spectrum. Let
$(\hat{X},\hat\D)$ be the
unbounded Kasparov $M(F,\A),F$ module of Proposition \ref{thmone}.
 Then for any unitary $u\in \A$ 
 such that {\bf $u^*[\D,u]$ is bounded and commutes
 with $\D$} we have the following equality of
index pairings with values in $K_0(F)$:
\bean \la [u^*],[(X,\D)]\ra&:=&{\rm Index}(Pu^*P)
=
\la [e_u]-\left[\bma 1 & 0\\ 0 & 0\ema\right],[(\hat{X},\hat\D)]\ra
\in K_0(F).\eean
Moreover, if $v\in \A$ is a  partial isometry, with
$vv^*,v^*v\in F$ and $v^*[\D,v]$ bounded and commuting with 
$\D$  we have
\bea\la [e_v]-\left[\bma 1 & 0\\ 0 & 0\ema\right],
[(\hat{X},\hat{\D})]\ra&=&-{\rm Index}(PvP:v^*vP(X)\to vv^*P(X))\in K_0(F)\nno
&=&{\rm Index}(Pv^*P:vv^*P(X)\to v^*vP(X))\in K_0(F).\label{pisomindex}\eea
 \end{prop}
We remark that the hypothesis that $\D$ and $v^*[\D,v]$ commute can be 
considerably relaxed (with considerable effort).
We will see later 
how this theorem assists us when $\bar{\mathcal A} =O_n$ and $O_n$
is equipped with its natural KMS state.
Before turning to this we compute $K_0(M(F,A))$ for the
examples we have in mind. Let $A=O_n$ and let $\A$ be any dense
smooth subalgebra such that the fixed point algebra for the modular
automorphism, $F$, is contained in $\A$. 
Using $K_1(\A)= K_1(F)=0$, the six term sequence in $K$-theory becomes
$$0\to K_0(\cM(F,\A))\to K_0(F)\to K_0(\A)\to K_1(\cM(F,\A))\to 0.$$
Now $K_0(F)=\Z[1/n]$ while $K_0(A)=\Z_{n-1}$, \cite{Da}. A careful analysis of
the map $K_0(F)\to K_0(\A)$ shows that it is  induced by
inclusion, \cite{CPR}. Since
$K_0(\A)=\{0,Id,2Id,...,(n-2)Id\}$ for the Cuntz algebra,  
this map is onto. Hence
$K_1(\cM(F,\A))=0$ and $K_0(\cM(F,\A))=(n-1)\Z[1/n]$.

\section{The modular spectral triple of the Cuntz algebras}

The Cuntz algebras do not possess a faithful gauge invariant
trace. There is however a unique state which is KMS for the gauge 
action, namely $\psi:=\tau\circ\Phi:O_n\to\C$,
where $\Phi:O_n\to F$ is the expectation and $\tau:F\to\C$ the
unique faithful normalised trace.   As the Cuntz algebras 
satisfy the hypotheses of \cite{pr}  (they are  graph algebras 
of a locally finite
graph with no sources), the generator of the gauge action $\D$
acting on the right $C^*$-$F$-module $X$ gives us a Kasparov module
$(X,\D)$. As with tracial graph algebras, we take
this class as our starting point.
However we immediately encounter a difficulty that there are no unitaries to
pair with, since $K_1(O_n)=0$. Nevertheless, there are many partial
isometries with range and source in the fixed point algebra ($O_n$ is
generated by such elements), so the APS pairing of the previous
section is available. 
We would like to compute a numerical pairing using a spectral triple
and we use the Kasparov module for this purpose.

Let $\HH=\LL^2(O_n)$ be the GNS Hilbert space given by the faithful state
$\psi=\tau\circ\Phi.$ That is, the inner product on $O_n$ is defined by
$\la a,b\ra=\psi(a^*b)=(\tau\circ\Phi)(a^*b).$
Then $\D$ extends to a self-adjoint unbounded operator on $\HH$, \cite{pr},
and we denote this closure by $\D$ from now on. The representation $\pi$
of $O_n$ on $\HH$ by left multiplication is bounded and
nondegenerate, and the dense subalgebra $\mbox{span}\{\pi(S_\mu S_\nu^*)\}$ 
is in the smooth domain of the derivation 
$\delta=\mbox{ad}(|\D|)$.
We denote the left action of an element $a\in O_n$ by $\pi(a)$ so that
$\pi(a)b=ab$ for all $b\in O_n.$ This distinction between elements of
$O_n$ as vectors in $\LL^2(O_n)$ and operators on $\LL^2(O_n)$ is
sometimes crucial. Thus we
see that the central algebraic structures of the gauge spectral
triple on a tracial graph algebra are mirrored in this
construction.

What differs  from the tracial situation is the
analytic information. We begin by obtaining some information about
the trace on $F$, the corresponding state on $O_n$ and the
associated modular theory.

\begin{lemma}\label{tracegens} The trace $\tau:F\to\C$ satisfies 
$\tau(S_\mu S^*_\nu)=\delta_{\mu,\nu}\frac{1}{n^{|\mu|}}.$
\end{lemma}
\begin{proof} First of all, we must have $|\mu|=|\nu|$ in order that
$S_\mu S_\nu^*\in F$, and then
\bean \tau(S_\mu S^*_\nu) & = &
\tau(S_\mu S^*_\mu S_\mu S^*_\nu)=\qquad\tau(S_\mu S_\nu^*S_\nu S_\nu^*)\nno
&=& \tau(S_\mu S_\nu^* S_\mu S^*_\mu)=\qquad\tau(S_\nu S_\nu^*S_\mu
S_\nu^*)\nno
&=&  \delta_{\mu,\nu}\tau(S_\mu
S^*_\mu)\ \ =\qquad\delta_{\mu,\nu}\tau(S_\nu S_\nu^*).\eean 
Thus whenever $|\mu|=|\nu|$ we have $\tau(S_\mu S_\mu^*)=\tau(S_\nu S_\nu^*)$.
  Since there are exactly
$n^k$ distinct $S_\mu$ all with orthogonal ranges so that 
$\sum_{|\mu|=k}S_\mu S_\mu^* =1$, the result follows.
\end{proof}

Let $\S$ first denote the operator
$a \mapsto a^*$ defined on $O_{nc}$ as a subspace of $\LL^2(O_n)$. 
The conjugate-linear adjoint of $\S$ exists, is denoted $\F$ and will be 
explicitly calculated on the subspace
$O_{nc}$ in the next lemma. It satisfies
 $$\F(S_{\mu}S_{\nu}^*)=n^{(|\mu|-|\nu|)}S_{\nu}S_{\mu}^*.$$
In particular, $\F$ is densely defined so that $\S$ is closable. So we use the same symbol 
$\S$ to denote the closure 
and also $\F$ will denote the closure of $\F$ restricted to $O_{nc}$.
 Then $\S$  has a polar decomposition as 
\ben \S=J\Delta^{1/2}=\Delta^{-1/2}J,\ \ \F=J\Delta^{-1/2}=\Delta^{1/2}J,\ \  
\;{\rm where}\;\;\;\Delta=\F\S,\een 
where $J$
is an antilinear map, $J^2=1$. The Tomita-Takesaki modular theory, \cite{KR},
shows that \ben \Delta^{-it}\pi(O_n)''\Delta^{it}=\pi(O_n)'',\ \ \
J\pi(O_n)''J^*=(\pi(O_n)'')',\een where $\pi(O_n)''$ is the weak closure of
the left action of $O_n$ on $\LL^2(O_n)$. However, all of these operators can be
explicitly calculated on the subspace $O_{nc}$ which is in fact a 
Tomita algebra.

\begin{lemma}\label{tomita} The  algebra $ O_{nc}={\rm span}\{S_\mu S_\nu^*\}$ 
with the inner product, $\la a|b\ra=\psi(a^*b)=\tau\circ\Phi(a^*b)$ 
arising from the state 
$\psi=\tau\circ\Phi$ is a Tomita algebra 
(except for the trivial difference that 
our inner product is linear in the second coordinate). 
\end{lemma}

\begin{proof} 
 Since the inner product on $O_{nc}$ comes from the GNS construction
given by the faithful state $\psi=\tau\circ\Phi$ 
the left action of $O_{nc}$ on itself is involutive, faithful (and hence 
isometric),
and nondegenerate. This takes care of Takesaki's Axioms (I), (II), (III)
for a Tomita algebra \cite{Ta}.
Next, as mentioned above, the operator $\S$ on $O_{nc}$ is just the 
mapping on $O_n$,
$a\mapsto \S(a)= a^*$, which, on generators, is
$\S(S_\mu S_\nu ^*)=S_\nu S_\mu ^*.$ We define the conjugate linear map $\F$
on generators (and extend by linearity) via:
$$\F(S_\mu S_\nu ^*)=n^{(|\mu|-|\nu|)}S_\nu S_\mu ^*.$$
To see that $\F$ is the adjoint of $\S$ it suffices to check the defining 
equation: $\la \S(a)|b\ra = \la \F(b)|a\ra$ on generators $a=S_\alpha S_\beta ^*$
and $b=S_\mu S_\nu ^*.$ Then, 
$$\la \S(a)|b\ra=\tau\circ\Phi(S_\alpha S_\beta ^* S_\mu S_\nu ^*)\;\;\;
{\rm while}
\;\;\;
\la \F(b)|a\ra=\tau\circ\Phi(S_\mu S_\nu ^* S_\alpha S_\beta ^*)n^{(|\mu|-|\nu|)}.$$
Now, if $|\mu|+|\alpha|-|\nu|-|\beta|\neq 0$ then both terms are
$0$ hence equal. While if $|\mu|+|\alpha|-|\nu|-|\beta|= 0,$ then
$$\la \S(a)|b\ra=\tau(S_\alpha S_\beta ^* S_\mu S_\nu ^*)\;\;\;{\rm while}
\;\;\;
\la \F(b)|a\ra=\tau(S_\mu S_\nu ^* S_\alpha S_\beta ^*)n^{(|\mu|-|\nu|)}.$$
In the second case where $|\beta|-|\mu|=|\alpha|-|\nu|$, we assume that
$|\beta|-|\mu|=|\alpha|-|\nu|\geq 0$ as the case $\leq 0$ is very similar.

Now, $S_\alpha S_\beta ^* S_\mu S_\nu ^*=0$ unless $\beta=\mu\lambda$, whence
$S_\alpha S_\beta ^* S_\mu S_\nu ^*=S_\alpha S_\lambda ^* S_\nu ^*$ and
$|\lambda|=|\beta|-|\mu|.$  Then since $S_\alpha S_\lambda ^* S_\nu ^*\neq 0$
we must have $\alpha=\nu\lambda,$ and hence we have
$$S_\alpha S_\beta ^* S_\mu S_\nu ^*=S_\alpha S_\alpha ^*\;\;\;{\rm where}\;\;\;
\beta=\mu\lambda\;\;\;{\rm and}\;\;\;\alpha=\nu\lambda\;\;\;{\rm and}\;\;\;
|\lambda|=|\beta|-|\mu|.$$
Similary, $S_\mu S_\nu ^* S_\alpha S_\beta ^*=0$ unless
$ S_\mu S_\nu ^* S_\alpha S_\beta ^*=S_\beta S_\beta ^*$
where
$\alpha=\nu\gamma,$ $\beta=\mu\gamma$ and 
$|\gamma|=|\alpha|-|\nu|.$
We note that since $|\lambda|=|\beta|-|\mu|=|\alpha|-|\nu|=|\gamma|$, we have that
the two expressions $S_\alpha S_\beta ^* S_\mu S_\nu ^*$ and 
$S_\mu S_\nu ^* S_\alpha S_\beta ^*$ are both nonzero at the same time with the
same condition: $\beta=\mu\lambda$ and $\alpha=\nu\lambda.$
Finally, $\tau(S_\alpha S_\beta ^* S_\mu S_\nu ^*)=\tau(S_\alpha S_\alpha ^*)
=n^{-|\alpha|}$ while $$\tau(S_\mu S_\nu ^* S_\alpha S_\beta ^*)n^{(|\mu|-|\nu|)}
=\tau(S_\beta S_\beta ^*)n^{(|\mu|-|\nu|)}=n^{-|\beta|}n^{(|\mu|-|\nu|)}=\cdots
=n^{-|\alpha|}.$$ Thus, $\F$ is the adjoint of $\S$ and so both are closable. This
takes care of Takesaki's Axiom (IX).

We immediately deduce that
$\Delta(S_\mu S_\nu ^*)=\F\S(S_\mu S_\nu ^*)=n^{(|\nu|-|\mu|)}S_\mu S_\nu ^*$ so that
$$\Delta^{1/2}(S_\mu S_\nu ^*)=n^{(1/2)(|\nu|-|\mu|)} S_\mu S_\nu ^*$$
and $J(S_\mu S_\nu ^*)=n^{(1/2)(|\mu|-|\nu|)} S_\nu S_\mu ^*$
so that $\S=J\Delta^{1/2},$ $\F=\Delta^{1/2}J,$
as required. Moreover, for all $z\in \C$ we have:
$$\Delta^z(S_\mu S_\nu ^*)=n^{z(|\nu|-|\mu|)} S_\mu S_\nu ^*$$
where for $w\in\C$ we take $n^w:=e^{wlog(n)}.$
We remark that each $S_\mu S_\nu ^*$ is an eigenvector of $\Delta$ for the
nonzero eigenvalue $n^{(|\nu|-|\mu|)},$ and so each eigenvalue has
infinite multiplicity. 

We quickly review Takesaki's remaining axioms for a Tomita algebra.
First, there is the un-numbered axiom that each $\Delta^z :O_{nc}\to O_{nc}$ 
is an algebra homomorphism. Clearly,
each $\Delta^z$ is a linear isomorphism, and it suffices to check multiplicativity
on the generators. This is a calculation based on the following fact:
$(S_\mu S_\nu ^*)(S_\alpha S_\beta ^*)=0$ unless {\bf either}
$|\nu|\geq |\alpha|$ and $\nu=\alpha\lambda$ where
$(S_\mu S_\nu ^*)(S_\alpha S_\beta ^*)=S_\mu S_{\beta\lambda}^*$ {\bf or}
$|\nu|\leq |\alpha|$ and $\alpha=\nu\gamma$ where
$(S_\mu S_\nu ^*)(S_\alpha S_\beta ^*)=S_{\mu\gamma} S_\beta ^*.$ We remind the 
reader that this axiom says that as operators on $O_n$:
$$\pi(\Delta^z(a))=\Delta^z\pi(a)\Delta^{-z}\;\;\;{\rm in\;\; particular},
\;\;\;\pi(\Delta^{it}(a))=\Delta^{it}\pi(a)\Delta^{-it}.$$

Axiom (IV): $\S(\Delta^z(a))=\Delta^{-\overline{z}}(\S(a))$ for all $a\in O_{nc}$
and all $z\in \C.$ This is a straightforward calculation.

Axiom (V): $\la\Delta^z(a)|b\ra=\la a|\Delta^{\overline{z}}(b)\ra$ for all 
$a,b\in O_{nc},$ and $z\in\C.$ Another easy calculation.

Axiom (VI): $\la \Delta(\S(a))|\S(b)\ra=\la b|a\ra$ for all $a,b\in O_{nc}.$ This is equivalent to $\la \F(a)|\S(b)\ra=\la b|a\ra.$

Axiom (VII): The function $z\mapsto \la a|\Delta^z(b)\ra$ is analytic on $\C$
for each $a,b\in O_{nc}.$ Again an easy calculation since our inner products
are {\bf linear in the second }variable. Finally we have:

Axiom (VIII): For each $t\in\R$ the subspace $(1+\Delta^t)(O_{nc})$ is dense in
$O_{nc}.$ In fact, each generator $S_\mu S_\nu ^*$ is an eigenvector of
$(1+\Delta^t)$ with positive eigenvalue: $1+n^{t(|\nu|-|\mu|)},$
and hence $(1+\Delta^t)(O_{nc})=O_{nc}.$ 
\end{proof}

\begin{lemma} The group of modular
automorphisms  of the von Neumann algebra $O_n^{\prime\prime}$
generated by the left action of
$O_n$ on $\LL^2(O_n)$ (which is the same as the von Neumann algebra 
generated
by the left action of $O_{nc}$ on $\LL^2(O_{nc})=\LL^2(O_n)$) is given on
the generators by  
\be \s_t(\pi(S_\mu S^*_\nu)):=\Delta^{it}\pi(S_\mu
S^*_\nu)\Delta^{-it}=\pi(\Delta^{it}(S_\mu S^*_\nu))
=n^{it(|\nu|-|\mu|)}\pi(S_\mu S^*_\nu).\label{eq:kms-flow}\ee
\end{lemma}
\begin{proof} This is a straightforward calculation obtained by evaluating
these operators on a generator $S_\alpha S_\beta ^*$ and using the
un-numbered Axiom that 
$\Delta^{it}$ is an algebra homomorphism on $O_{nc}.$
\end{proof}



\begin{rems*}
A special case of the KMS condition on the modular automorphism group of the state
$\psi$, \cite{Ta}, (for $t=i$) is the following: $\psi(xy)=\psi(\s_i(y)x)$
for $x,y\in\pi(O_n).$ The proof is elementary:
$$\tau\circ\Phi(xy)=\la x^*|y\ra=\la \S(x)|y\ra=\la \F(y)|x\ra=
\tau\circ\Phi(\S\F(y)x)=\tau\circ\Phi(\Delta^{-1}(y)x)
=\tau\circ\Phi(\s(y)x).$$
From now on we refer to this as the KMS condition for the state $\psi.$
\end{rems*}

\begin{cor} With $O_n$ acting on $\mathcal{H}:=\LL^2(O_n)$ we let $\D$ 
be the generator of the natural unitary implementation of the gauge action of 
$\T^1$ on $O_n.$ Then we have
\ben \Delta=n^{-\D}\  {\rm  or  }\
e^{it\D}=\Delta^{-it/log n}.\een
\end{cor}


To continue, we recall the underlying right $C^*$-$F$-module, $X$,
which is the completion of $O_n$ for the norm $\Vert
x\Vert_X^2=\Vert\Phi(x^*x)\Vert_F$.

\begin{lemma}
Any $F$-linear endomorphism $T$ of the module $X$
which preserves the copy of $O_n$ inside $X$, extends uniquely to a bounded 
operator on the Hilbert space $\HH=\LL^2(O_n).$
\end{lemma}
\begin{proof}
For any $x\in X$ we have, by \cite[Proposition 1.2]{L},
$(Tx|Tx)_R\leq \Vert T\Vert^2_{End} (x|x)_R$ in $F^+.$
Letting $\Vert T\Vert_\infty$ denote the operator norm on  $\HH$ we estimate
using $x\in O_n:$
$$\Vert T\Vert_\infty^2 =\sup_{\Vert x\Vert_{\HH}\leq 1}
\la Tx|Tx\ra_{\HH}=\sup_{\Vert x\Vert_{\HH}\leq 1}
\tau((Tx|Tx)_R) \leq \sup_{\Vert x\Vert_{\HH}\leq 1}
\Vert T\Vert^2_{End} \tau((x|x)_R)=\Vert T\Vert^2_{End}.$$
\end{proof}
In particular, the finite rank endomorphisms of the pre-$C^*$ module
$O_{nc}$ (acting on the left) 
satisfy this condition, and we denote the algebra of all these
endomorphisms by $End_F^{00}(O_{nc})$.

\begin{prop}\label{tildetrace} Let $\cn$ be the von Neumann algebra
$  \cn=(End^{00}_F(O_{nc}))'',$
where we take the commutant inside $\B(\HH)$. Then $\cn$ is 
semifinite, and there exists a faithful,
semifinite, normal trace $\tilde\tau:\cn\to\C$ such that for all rank
one endomorphisms $\Theta^{R}_{x,y}$ of $O_{nc}$, 
$$\tilde\tau(\Theta^{R}_{x,y})=(\tau\circ\Phi)(y^*x),\ \ \ x,y\in O_{nc}.$$
In addition, $\D$ is affiliated to $\cn$ and $O_n$, acting on the left on
$X$, is a subalgebra of $\cn$.
\end{prop}

\begin{proof}  We define $\tilde\tau$ as a supremum of an increasing
sequence of vector states, as in \cite{pr}, which ensures
that $\tilde\tau$ is normal. First for $|\mu|\neq 0$ we define for $T\in\cn$
$$\omega_\mu(T):=\la S_\mu,TS_\mu\ra+\frac{1}{n^{|\mu|}}\la
S_\mu^*,TS_\mu^*\ra.$$
Together with $\omega_1(T):=\la 1,T1\ra$, this gives a collection of
positive vector states on $\cn$. We define
$$\tilde\tau(T)=\omega_1(T)+\lim_{L\nearrow}\sum_{\mu\in L}\omega_\mu(T),$$
where $L$ ranges over the finite subsets of the finite path space
$E^*$ of the graph underlying $O_n$. With this definition, the proof
in \cite[Lemma 5.11]{pr} can be applied almost verbatim to this case. 
The only real change in the proof occurs on page 121 of \cite{pr}: the line 
before the phrase ``the last inequality following'' should be replaced by:
$$=\Vert T\Vert\sum_{s(\mu)=v,|\mu|=k} \tau(p_{r(\mu)})=
\Vert T\Vert n^k \tau(p_v)<\infty.$$

Rather than repeat the proof here, we simply observe for the reader's
benefit that to check the trace
property (on endomorphisms) only requires that $\tau$ is a trace on 
$F$, not all of $O_n$. Here is the formal
calculation for rank one operators:
\bean \tilde\tau(\Theta^{R}_{w,z}\Theta^{R}_{x,y})&=&
\tilde\tau(\Theta^{R}_{w(z|x),y})
=\tau\circ\Phi(y^*w(z|x))=\tau((y|w(z|x)))\nno
&=&\tau((y|w)(z|x))=\tau((z|x)(y|w))\nno
&=&\tilde\tau(\Theta^{R}_{x(y|w),z})=
\tilde\tau(\Theta^{R}_{x,y}\Theta^{R}_{w,z}).\eean

Next we must show that $\D$ is affiliated to $\cn$. However, we
have already noted that the spectral projections of $\D$ are finite
sums of rank one endomorphisms of $X_c$, in the paragraph immediately preceding Theorem 2.8. This
proves 
the claim. That $A_c$ embeds in $\cn$ follows from Lemma
\ref{finrank} and the fact that the $\Phi_k$ sum to the identity. Since $A$
is the unique $C^*$-completion of $A_c$ we see that $\pi$ embeds $A$ in $\cn.$
\end{proof}

Unfortunately, in contrast to the situation in \cite{pr},
this trace is not what we
need for defining summability.
This can be seen from the following calculations. For $k\geq 0$ \bean
\tilde\tau(\Phi_k)=\tilde\tau(\sum_{|\rho|=k}\Theta^{R}_{S_\rho,S_\rho})
=\tau(\sum_{|\rho|=k}(S_\rho|S_\rho))
=\tau(\sum_{|\rho|=k}S_\rho^*S_\rho)=\sum_{|\rho|=k}1
=n^k.\eean Similarly, for $k<0$ we have $\tilde\tau(\Phi_k)=n^k.$
Hence with respect to this trace we cannot expect $\mathcal D$ to satisfy
any summability criterion.

\begin{defn} 
We define a new weight on $\cn^+$: let $T\in\cn^+$ then
$\tau_\Delta(T):=\sup_N\tilde\tau(\Delta_N T)$ where 
$\Delta_N=\Delta(\sum_{|k|\leq N}\Phi_k).$ 
\end{defn}

{\bf Remarks}. Since $\Delta_N$ is 
$\tilde\tau$-trace-class, we see that $T\mapsto\tilde\tau(\Delta_N T)$
is a normal positive linear functional on $\cn$ and hence 
$\tau_\Delta$ is a normal weight on $\mathcal N^+$ which is easily seen to 
be faithful and semifinite. 

We now give another way to define 
$\tau_\Delta$ which is not only conceptually useful but also makes a
number of important properties straightforward to verify.

\noindent {\bf Notation}. Let $\cm$ be the relative commutant
in $\mathcal N$ of the operator $\Delta$. Equivalently, $\cm$ is the
relative commutant of the set of spectral projections $\{\Phi_k | k\in\Z\}$  
Clearly,
$\cm=\sum_{k\in\Z}\;\Phi_k\cn\Phi_k.$

\begin{defn} As $\tilde\tau$ restricted to each  $\Phi_k\cn\Phi_k$ is a 
faithful finite trace with $\tilde\tau(\Phi_k)=n^k$
we define $\widehat\tau_k$ on $\Phi_k\cn\Phi_k$ to be $n^{-k}$
times the restriction of $\tilde\tau.$ Then, 
$\widehat\tau:=\sum_k\widehat\tau_k$
on $\cm=\sum_{k\in\Z}\Phi_k\cn\Phi_k$ is a faithful normal semifinite
trace $\widehat\tau$ with $\widehat\tau(\Phi_k)=1$ for all $k.$ 
\end{defn}
We use $\widehat\tau$ to 
give an alternative  expression for $\tau_\Delta$ below
This alternative might be avoidable
but at the expense of a detailed use of \cite{PT}. However, 
(see the bottom of 
page 61 of \cite{PT}), {\bf the semifiniteness} of $\tau_\Delta$ 
restricted to $\cm$
{\bf depends on} the existence of a normal $\tau_\Delta$-invariant projection 
(such as $\Psi$ defined below) from $\cn$ onto $\cm.$

\begin{lemma}\label{commutes} An element $m\in\cn$ is in $\cm$ if and
only if it is in the fixed point algebra of the action, $\s_t^{\tau_\Delta}$
on $\cn$ defined for $T\in\cn$ by 
$\s_t^{\tau_\Delta}(T)=\Delta^{it}T\Delta^{-it}.$ Both  $\pi(F)$ and the 
projections $\Phi_k$ belong to $\cM$. The map $\Psi :\cn\to\cm$ defined by
$\Psi(T)=\sum_k \Phi_k T\Phi_k$ is a conditional expectation onto $\cm$
and $\tau_{\Delta}(T)=\widehat\tau(\Psi(T))$ for all $T\in\cn^+.$ That is,
$\tau_{\Delta}= \widehat\tau\circ\Psi$ so that $\widehat\tau(T)=\tau_{\Delta}(T)$ for 
all $T\in\cm^+.$ Finally, if one of $A,B\in\cm$ is $\widehat\tau$-trace-class
and $T\in\cn$ then $\tau_\Delta(ATB)=\tau_\Delta(A\Psi(T)B)=
\widehat\tau(A\Psi(T)B).$
\end{lemma}


\begin{proof}
The first two statements are immediate. Also, the fact that $\Psi$ is a 
unital norm one projection of $\cn$ onto $\cm$ (and hence
a normal conditional expectation by Tomiyama's theorem \cite{T}) is clear. Only 
the last assertions of the 
Lemma need proof. To this end let $T\in\cn^+,$ then
\bean
\tau_{\Delta}(T)&=&\sup_N\tilde\tau(\Delta_N T)=
\sup_N \tilde\tau(\Delta(\sum_{|k|\leq N}\Phi_k)T)=
\sup_N \tilde\tau(\sum_{|k|\leq N}\Delta\Phi_k T)\\
&=&\sup_N \tilde\tau(\sum_{|k|\leq N}n^{-k}\Phi_k T)=
\sup_N \sum_{|k|\leq N}n^{-k}\tilde\tau(\Phi_k T\Phi_k)=
\sum_{k\in\Z}n^{-k}\tilde\tau(\Phi_k T\Phi_k)=\widehat\tau(\Psi(T)).
\eean
Hence if $T\in\cm$ then $\widehat\tau(T)=\widehat\tau(\Psi(T))=\tau_\Delta(T).$
Finally the last statement follows from the fact that $\Psi(ATB)=A\Psi(T)B$
by Tomiyama's Theorem \cite{T}.
\end{proof}

\begin{lemma} The modular automorphism group $\s_t^{\tau_\Delta}$ of
$\tau_\Delta$ is inner and given by
$\s_t^{\tau_\Delta}(T)=\Delta^{it}T\Delta^{-it}$. The weight
$\tau_\Delta$ is a KMS weight for the group $\s_t^{\tau_\Delta}$, and 
$\s_t^{\tau_\Delta}|_{O_n}=\s_t^{\tau\circ\Phi}.$
\end{lemma}

\begin{proof} This follows from: \cite[Thm 9.2.38]{KR}, which gives
us the KMS properties of $\tau_\Delta$: the modular group is inner since
$\Delta$ is affiliated to $\cn.$ The final
statement about the restriction of the modular group to $O_n$ is clear.
\end{proof}


The reward for having sacrificed a trace on $\cn$ for a trace on $\cM$
is the following.

\begin{lemma}\label{tracesplit}
Suppose $g$ is a function on $\R$ such that $g(\D)$ is $\tau_\Delta$ trace-class
in $\cm$, then for all $f\in F$ we have
$$\tau_\Delta(\pi(f)g(\D))=\tau_\Delta(g(D))\tau(f)=
\tau(f)\sum_{k\in\Z} g(k).$$
\end{lemma}

\begin{proof}
First note that $\tau_\Delta(g(\D))=\widehat\tau(\sum_{k\in\Z}g(k)\Phi_k)
=\sum_{k\in\Z}g(k)\widehat\tau(\Phi_k)=\sum_{k\in\Z}g(k).$
Now, $$\tau_\Delta(\pi(f)g(\D))=\widehat\tau(\pi(f)\sum_{k\in\Z}g(k)\Phi_k)
=\sum_{k\in\Z}g(k)\widehat\tau(\pi(f)\Phi_k)$$
$$=\sum_{k\in\Z}g(k)\widehat\tau_k(\pi(f)\Phi_k)
=\sum_{k\in\Z}g(k)n^{-k}\tilde\tau(\pi(f)\Phi_k).$$
So it suffices to see for each $k\in\Z$, we have 
$\tilde\tau(\pi(f)\Phi_k)=n^k\tau(f).$

For all $f\in F$, $f$ is a norm limit of finite sums of terms like 
$S_\al S_\beta^*$, $|\al|=|\beta|=r$. So we compute for $f=S_\alpha S_\beta^*.$
Recall that we have the formulae
$$\Phi_k=\sum_{|\mu|=k}\Theta^{R}_{S_\mu,S_\mu},\ \
k>0,\quad\quad
\Phi_k=n^{-k}\sum_{|\mu|=|k|}\Theta^{R}_{S_\mu^*,S_\mu^*},\ \ k<0.
\quad\quad \mbox{and }\Phi_0=\Theta^{R}_{1,1}$$ where,
with $\mu$ the path of length zero,  we are using the notation $1=S_\mu$.

First for $k\geq 0$
$$\tilde\tau(\pi(f)\Phi_k)=
\tilde\tau(\pi(f)\sum_{|\mu|=k}\Theta^{R}_{S_\mu,S_\mu})=
\tilde\tau(\sum_{|\mu|=k}\Theta^{R}_{fS_\mu,S_\mu})=
\sum_{|\mu|=k}\tau\circ\Phi({S_\mu^* f S_\mu})$$
$$=\sum_{|\mu|=k}\tau({S_\mu^* S_\alpha S_\beta^* S_\mu})=n^{k-r}\delta_{\alpha,
\beta}=n^k\frac{1}{n^{|\alpha|}}\delta_{\alpha,\beta}=n^k\tau(S_\alpha S_\beta^*)
=n^k\tau(f).$$
A similar calculation holds for $k<0$ using the other formula for $\Phi_k$
in this case. Since all $f\in F_c$ are
linear combinations $S_\al S_\beta^*$, $|\al|=|\beta|$, we get for all
$f\in F_c,$ the formula 
$$\tau_\Delta(\pi(f)g(\D))=\tau_\Delta(g(D))\tau(f)=
\sum_{k\in{\Z}}g(k)\tau(f).$$
Now, the right hand side is a norm-continuous function of $f.$
To see that the left side is norm-continuous we do it in more generality.
Let $T\in\cn$, then since $\widehat\tau$ is a trace on $\cm$ we get:
$$|\tau_{\Delta}(Tg(\D))|=|\widehat\tau(\Psi(Tg(\D))|=
|\widehat\tau(\Psi(T)g(\D))|\leq
\Vert \Psi(T)\Vert \widehat\tau(|g(\D)|)\leq
\Vert T\Vert \widehat\tau((|g(\D)|)=
\Vert T\Vert \tau_{\Delta}(|g(\D)|).$$
That is the left hand side is norm-continuous in $T$ and so we have 
the formula:
$$\tau_\Delta(\pi(f)g(\D))=\tau_\Delta(g(\D))\tau(f)=\sum_{k\in{\Z}}g(k)\tau(f)$$ for 
all $f\in F.$
\end{proof}

\begin{rems*}
The inequality above clearly holds in more generality. That is, if $T\in\cn$ 
and $B\in \LL^1(\cm,\tau_\Delta)$ then:
\begin{eqnarray}\label{cauchy}
|\tau_\Delta(TB)|\leq \Vert T\Vert_\infty \tau_\Delta(|B|)=
\Vert T\Vert_\infty \Vert B\Vert_1. 
\end{eqnarray}
\end{rems*}

\begin{prop}\label{dixycomp} We have 
$(1+\D^2)^{-1/2}\in\LL^{(1,\infty)}(\cM,\tau_\Delta)$. That is,
$\tau_\Delta((1+\D^2)^{-s/2})<\infty$ for all $s>1.$ Moreover, for
all $f\in F$
$$\lim_{s\to 1^+}(s-1)\tau_{\Delta}(\pi(f)(1+\D^2)^{-s/2})=2\tau(f)$$
so that $\pi(f)(1+\D^2)^{-1/2}$ is a measurable operator in the sense of 
\cite{C}.
\end{prop}

\begin{proof}
Let $s>1$. Then
$\tau_\Delta((1+\D^2)^{-s/2})=\widehat\tau(\sum_{k\in{\Z}}(1+k^2)^{-s/2}\Phi_k)
=\sum_{k\in{\Z}}(1+k^2)^{-s/2}.$ Hence, $(1+\D^2)^{-s/2}$ is 
$\tau_\Delta$-trace-class in $\cm$ for all $Re(s)>1$ and
$$\lim_{s\to 1^+}(s-1)\tau_\Delta((1+\D^2)^{-s/2})=2.$$
By Lemma \ref{tracesplit} we have the equality:
$$\tau_\Delta(\pi(f)(1+\D^2)^{-s/2})=\sum_{k\in{\Z}}(1+k^2)^{-s/2}\tau(f)$$ for 
all $f\in F.$ Hence, 
$$ \lim_{s\to 1^+}(s-1)\tau_\Delta(\pi(f)(1+\D^2)^{-s/2})=2\tau(f)$$
and $\pi(f)(1+\D^2)^{-1/2}$ is measurable, for all $f\in F$.
\end{proof}

We wish to extend our conclusions about $\tau_\Delta$ and
$\lim_{s\to 1^+}(s-1)\tau_\Delta(T(1+\D^2)^{-s/2})$ to the whole von Neumann
algebra $\cn$. Unfortunately, these limits do not exist for general
$T\in\cn$ and we are forced to consider generalised limits as in the Dixmier 
trace theory.
\begin{defn}
 Let $\tilde\omega$
be a state on $\mathcal L^\infty(\R_+)$ which satisfies the 
condition that if $g\in\mathcal L^\infty(\R_+)$ is real-valued then 
$$\liminf_{t\to\infty} g(t) \leq \tilde\omega(g)
\leq \limsup_{t\to\infty} g(t).$$
Clearly any such state is identically $0$ on $C_0(\R_+)$ and also on
any function which is essentially compactly supported. Moreover, if $g$
has a limit at $\infty$ then $\tilde\omega(g)=\lim_{t\to\infty} g(t).$
We define
$$\tilde\omega\!-\!\lim_{t\to\infty} g(t):= \tilde\omega(g).$$
\end{defn}

The existence of such states (with even more properties) 
can be found in \cite[Corollary 1.6]{CPS2}. In order to evaluate 
such states $\tilde\omega$ on functions $g$ of the form 
$s\mapsto (s-1)\tau_{\Delta}(T(1+\D^2)^{-s/2})$ for $s>1$ 
we need to do a translation: let $s=1+1/r$ then letting $s\to 1^+$ is the 
same as letting $r\to\infty.$ And we consider
$$(s-1)\tau_{\Delta}(T(1+\D^2)^{-s/2})=\frac{1}{r}
\tau_{\Delta}\left(T\left((1+\D^2)^{-1/2}\right)^{1+1/r}\right).$$
Of course, the limit of the left hand side of this equation exists
as $s\to 1^+$ if and only if the limit of the right hand side exists as 
$r\to\infty$
and in this case they are equal.

 {\bf Abuse of notation:} 
$$\tilde{\omega}\!-\!\lim_{r\to\infty}\frac{1}{r}
\tau_{\Delta}\left(T\left((1+\D^2)^{-1/2}\right)^{1+1/r}\right)\;\;\;
{\bf becomes}\;\;\;\tilde{\omega}\!-\!\lim_{s\to 1^+}(s-1)
\tau_{\Delta}(T(1+\D^2)^{-s/2}).$$
\begin{rems*}
Since $\tau_{\Delta}(T(1+\D^2)^{-s/2})=\widehat\tau(\Psi(T)(1+\D^2)^{-s/2})$
these generalised traces are taking place completely inside $\cm$ with respect
to the trace $\widehat\tau.$ That is,
we are in the now well-understood semifinite situation.
\end{rems*}

\begin{prop}\label{funnystate} Let $\tilde\omega$
be a state on $\mathcal L^\infty(\R_+)$ which satisfies the 
condition above. The functional $\widehat\tau_\omega$ on $\cn$  defined by
$$ \widehat\tau_\omega(T)=
\frac{1}{2}\tilde{\omega}\!-\!\lim_{s\to 1^+}(s-1)
\tau_{\Delta}(T(1+\D^2)^{-s/2})$$
is a state. 
For $T=\pi(a)\in \pi(O_n)\subset \cn$ the following (ordinary) 
limit exists and
$$\widehat\tau_\omega(\pi(a))=\frac{1}{2}\lim_{s\to 1^+}(s-1)
\tau_\Delta(\pi(a)(1+\D^2)^{-s/2})=\tau\circ\Phi(a),$$
the original KMS state $\psi=\tau\circ\Phi$ on $O_n.$
\end{prop}

\begin{proof} First we observe that $\tau_\Delta(T(1+\D^2)^{-s/2})$ is finite
for $s>1$ for all $T\in\cn$, since we showed 
in the proof of the previous Proposition that:
$$|\tau_\Delta(T(1+\D^2)^{-s/2})| \leq \Vert T\Vert
\tau_\Delta((1+\D^2)^{-s/2}).$$
Therefore, $(s-1)\tau_\Delta(T(1+\D^2)^{-s/2})$ is uniformly bounded
and so the generalised limit exists as $s\to 1^+$. It is easy to see that
this functional is positive on $\cn^+$ and by the previous proposition
$\widehat\tau_\omega(1)=1$, so that $\widehat\tau_\omega$ is a state on $\cn.$

Now, one easily checks by calculating on generators that for 
$\pi(a)\in \pi(O_{nc})$,
$\Psi(\pi(a))=\pi(\Phi(a))\in \pi(F_c)$ and since $\Psi$ is norm continuous we 
have that
$\Psi(\pi(a))=\pi(\Phi(a))\in \pi(F)$ for all $a\in O_n.$ Thus by 
Proposition \ref{dixycomp}, 
for $a\in O_n$ (letting $f=\Phi(a)$) we have
$$\;\tau\circ\Phi(a)=\frac{1}{2}
\lim_{s\to 1^+}(s-1)\tau_\Delta(\pi(\Phi(a))(1+\D^2)^{-s/2})$$
$$=\frac{1}{2}\lim_{s\to 1^+}(s-1)\tau_\Delta(\Psi(\pi(a))(1+\D^2)^{-s/2})=
\;\widehat\tau_\omega(\pi(a)).$$

Of course $\widehat\tau_\omega$ is a true Dixmier-trace since 
for $T\in \cn$ with $T\geq 0$, we have $\Psi(T)\in\cM$,
$\Psi(T)\geq 0$, and
$(1+\D^2)^{-1/2}\in\LL^{(1,\infty)}(\cM,\widehat\tau)$. Thus 
$$\tilde{\omega}\!-\!\lim_{r\to\infty}
\frac{1}{r}\tau_\Delta(T((1+\D^2)^{-1/2})^{1+1/r})=
\tilde\omega\!-\!\lim_{r\to\infty}
\frac{1}{r}\widehat\tau(\Psi(T)((1+\D^2)^{-1/2})^{1+1/r})$$
and the right hand side is a true Dixmier-trace on the semifinite algebra
$\cm$ provided we choose $\tilde\omega$ as in \cite[Theorem 3.1]{CPS2}. 
\end{proof}

We summarise our construction to date.\\
0) We have a $*$-subalgebra $\A=O_{nc}$ of the Cuntz algebra faithfully 
represented in $\cn$ with the latter acting on the  Hilbert space  
$\HH=\LL^2(O_{n},\psi)$,\\
1) there is a faithful normal semifinite weight $\tau_\Delta$
on $\cn$ such that the modular automorphism group of $\tau_\Delta$ is an
inner automorphism group $\tilde\sigma$ of $\cn$ with
$\tilde\sigma|_\A=\s$, \\
2) $\tau_\Delta$ restricts to a faithful semifinite trace $\widehat\tau$ 
on $\cM=\cn^\s$, with a faithful normal projection $\Psi: \cn\to\cm$
satisfying $\tau_\Delta=\widehat\tau\circ\Psi$ on $\cn.$\\
3) With $\D$ the generator of the one parameter group implementing 
$\sigma$ on $\HH$ we have:\\
 $[\D,\pi(a)]$ extends to a bounded operator (in $\cn$) for all $a\in\A$
and for $\lambda$ in the resolvent set of $\D$,
$(\lambda-\D)^{-1} \in \K(\cM,\tau_\Delta)$, where  
$\K(\cM,\tau_\Delta)$ is the ideal of compact operators in $\cM$ relative to
$\tau_\Delta$. In particular, $\D$ is affiliated to $\cM$.



\noindent{\bf Terminology/Definition}. 
The triple $(O_{nc},\HH,\D)$ along with $\cn,\ \tau_\Delta$ 
constructed in this section satisfying properties (0) to (3) above
we will refer to as a (unital) {\bf modular spectral triple}. For matrix 
algebras $\A= O_{nc}\otimes M_k$ over $O_{nc}$,
$(O_{nc}\otimes M_k,\HH\otimes M_k ,\D\otimes Id_k)$
is also a modular spectral triple in the obvious fashion. 
In work in progress we have found that this structure arises in other examples and appears to be a  quite general phenomenon.


%


We need some technical lemmas
for the discussion in the next Section. A function $f$ from a complex domain 
$\Omega$ into a Banach space $X$ is called {\bf holomorphic} if it is complex
differentiable in norm on $\Omega.$
\begin{lemma}

(1) Let $\mathcal B$ be a $C^*$-algebra and let 
$T\in\mathcal B^+.$ The mapping $z\mapsto T^z$ is holomorphic (in operator norm)
in the half-plane $Re(z)>0.$\\
(2) Let $\mathcal B$ be a von Neumann algebra with faithful normal 
semifinite trace $\phi$ and let $T\in\mathcal B^+$ be in 
$\LL^{(1,\infty)}(\mathcal B,\phi).$ Then, the mapping $z\mapsto T^z$ is 
holomorphic (in trace norm) in the half-plane $Re(z)>1.$\\
(3) Let $\mathcal B$, and $T$ be as in item (2) and let $A\in \mathcal B$
then the mapping $z\mapsto \phi(A T^z)$ is holomorphic for $Re(z)>1.$
\end{lemma}

\begin{proof}
To see item (1) we assume without loss of generality that $||T||\leq 1.$ We
fix $z_0\in\C$ with $Re(z_0)>0,$ and fix $R>0$ with $R<Re(z_0)$ so that the
circle $C: z=z_0 + Re^{i\theta}\;\;{\rm for}\;\;\theta\in[0,2\pi]$ lies in the
half-plane $Re(z)>0.$ Temporarily we fix $t\neq 0$ in the spectrum of $T$
so that $t\in (0,1].$ Now with $|z-z_0|<(1/2)R$ we apply the complex version of
Taylor's theorem to the function $z\mapsto t^z$ 
(see \cite[Theorem 8, pp125-6]{Ahl}) and get:
$$\frac{t^z-t^{z_0}}{z-z_0} - t^{z_0}Log(t)=f_2(z)(z-z_0)\;\;\;{\rm where}\;\;\;
f_2(z)=\frac{1}{2\pi i}\int_C \frac{t^w dw}{(w-z_0)^2(w-z)}.$$
So with $|z-z_0|<(1/2)R$ we get the estimate:
$$|f_2(z)|\leq\frac{\max_C|t^w|\cdot R}{R^2\cdot(1/2)R}
\leq\frac{2t^{(Re(z_0)+R)}}{R^2}\leq\frac{2}{R^2}.$$
Therefore,
$$\left|\frac{t^z-t^{z_0}}{z-z_0} - t^{z_0}Log(t)\right|
\leq\frac{2}{R^2}|z-z_0|.$$
Since this is true for all nonzero $t$ in the spectrum of $T$ we have:
$$\left\Vert\frac{1}{z-z_0}(T^z-T^{z_0})-T^{z_0}Log(T)\right\Vert_\infty
\leq\frac{2}{R^2}|z-z_0|.$$
That is $d/dz(T^z)=T^z Log(T)$ for $Re(z)>0$ with the limit existing in operator
norm.

To see item (2) we fix $z_0$ with $Re(z_0)>1$ and then fix $\epsilon$ 
sufficiently small so that $Re(z_0 -(1+\epsilon))=Re(z_0) -(1+\epsilon)>0.$
Then $T^{(1+\epsilon)}$ is trace-class, and this factor converts the operator
norm limits below into trace norm limits:
\bean
&&||\cdot||_1\lim_{z\to z_0}\frac{1}{z-z_0}(T^z-T^{z_0})=
||\cdot||_1\lim_{z\to z_0}T^{(1+\epsilon)}
\frac{1}{z-z_0}(T^{z-(1+\epsilon)}-T^{z_0-(1+\epsilon)})\\
&=&T^{(1+\epsilon)}\left(||\cdot||_\infty\lim_{z\to z_0}
\frac{1}{z-z_0}(T^{z-(1+\epsilon)}-T^{z_0-(1+\epsilon)})\right)\\
&=&T^{(1+\epsilon)}\left(||\cdot||_\infty\lim_{z\to z_0}
\frac{1}{(z-(1+\epsilon))-(z_0-(1+\epsilon))}
(T^{z-(1+\epsilon)}-T^{z_0-(1+\epsilon)})\right)\\
&=&T^{(1+\epsilon)}
(T^{z_0-(1+\epsilon)}Log(T)=T^{z_0} Log(T).
\eean
Item (3) follows from item (2) and inequality \eqref{cauchy}:
$|\phi(AB)|\leq ||A||_\infty||B||_1$ if $B$ is $\phi$-trace-class.
\end{proof}
\begin{lemma}\label{phiclosed}
In these modular spectral triples $(\A,\HH,\D)$ for matrices over the
 Cuntz algebras  we have
$(1+\D^2)^{-s/2}\in\LL^1(\cM,\tau_\Delta)$ for all
$s> 1$ and for $x\in \cn,$ 
$\tau_\Delta(x(1+\D^2)^{-r/2})$ is holomorphic for 
$Re(r)>1$ and we have for $a\in O_{nc}$
$\tau_\Delta([\D,\pi(a)](1+\D^2)^{-r/2})=0,$ for $Re(r)>1.$
\end{lemma}
\begin{proof} Since the eigenvalues for $\D$ are precisely the set of
integers, and the projection $\Phi_k$ on the eigenspace with eigenvalue $k$
satisfies $\tau_\Delta(\Phi_k)=1,$ it is clear that 
$(1+\D^2)^{-s/2}\in\LL^1(\cM,\tau_\Delta).$
Now, $\tau_\Delta(x(1+\D^2)^{-r/2})=
\widehat\tau(\Psi(x)(1+\D^2)^{-r/2})$ is holomorphic for $Re(r)>1$
by item (3) of the previous lemma.

To see the last statement, we observe that 
$\tau_\Delta([\D,\pi(a)](1+\D^2)^{-r/2})=
\tau_\Delta(\Psi([\D,\pi(a)])(1+\D^2)^{-r/2}),$ so it suffices to see that
$\Psi([\D,\pi(a)])=0$ for $a\in \A=O_{nc}.$   To this end, 
let $a=S_\alpha S_\beta^*$
be one of the linear generators of $O_{nc}.$  Then by calculating the action
of the operator $\D\pi(S_\alpha S_\beta^*)$ on the linear generators
$S_\mu S_\nu^*$ of the Hilbert space, $\HH$, we obtain:
$$\D\pi(S_\alpha S_\beta^*)=(|\alpha|-|\beta|)\pi(S_\alpha S_\beta^*)
+\pi(S_\alpha S_\beta^*)\D\;\;\;{\rm that\;\; is}\;\;\;
[\D,\pi(S_\alpha S_\beta^*)]=(|\alpha|-|\beta|)\pi(S_\alpha S_\beta^*).$$
More generally,
$$[\D,\pi(\sum_{i=1}^m c_i S_{\alpha_i} S_{\beta_i}^*)]=
\sum_{i=1}^m c_i(|\alpha_i|-|\beta_i|)\pi(S_{\alpha_i} S_{\beta_i}^*).$$
If we apply $\Psi$ to this equation, we see that 
$\Psi(\pi(S_{\alpha_i} S_{\beta_i}^*))=\pi(\Phi(S_{\alpha_i} S_{\beta_i}^*))=0$
whenever $(|\alpha|-|\beta|)\neq 0,$ and so the whole sum is $0$. We also
observe that $[\D,\pi(a)]\in \pi(O_{nc})$ for all $a\in O_{nc}.$ This is not
too surprising since $\D$ is the generator of the action $\gamma$ of $\T$ on
$O_n.$
\end{proof}

In the remainder of this paper we will shed some light on the
cohomological significance of these modular spectral triples. Just as
ordinary $\B(\HH)$ spectral triples represent $K$-homology classes,
\cite{C,CPRS1}, 
and semifinite spectral triples represent $KK$-classes, \cite{NR},
modular spectral triples provide analytic representatives
of some $K$-theoretic type data which we now describe.

\section{Modular $K_1$}

In this Section we introduce elements of $\mathcal A$ that will have a
well defined pairing with our Dixmier functional
 $\widehat\tau_{\omega}$.
Following \cite{HR} we say that a unitary (invertible, projection,...) 
{\bf in} $M_n(A)$ for some $n$ {\bf is}
a unitary (invertible, projection,...) {\bf over} $A$.

\begin{defn} Let $\A$ be a unital $*$-algebra and $\s:\A\to \A$ an algebra
automorphism such that 
$ \s(a)^*=\s^{-1}(a^*).$
We say that $\s$ is a {\bf regular automorphism}, \cite{KMT}.
\end{defn}

{\bf Remark}. The automorphism $\s(a):=\Delta^{-1}(a)$ of a modular
spectral triple is regular. This follows from AXIOM IV of Lemma \ref{tomita}:
$$(\s(a))^*=(\Delta^{-1}(a))^*=\S(\Delta^{-1}(a))=\Delta(\S(a))=\Delta(a^*)
=\s^{-1}(a^*).$$

\begin{defn} Let $u$ be a unitary over the $*$-algebra $\A$, and 
$\s:\A\to \A$ a regular
automorphism with fixed point algebra $F=\A^\s$. We say that $u$
satisfies the {\bf modular condition} with respect to $\s$ 
if both the operators
$ u\s(u^*)$ and $u^*\s(u)$
are matrices over the algebra $F$. We denote by
$U_\s$ the set of {\bf modular unitaries}. Of course, any unitary over $F$
is a modular unitary. 
\end{defn}

Here we are  thinking of the case $\s(a)=\Delta^{-1}(a)$, 
where $\Delta$ is the modular operator for some weight on
$A$. Again, to avoid confusion, we remind the reader that as operators we have:
$$\pi(\s(a))=\pi(\Delta^{-1}(a))=\Delta^{-1}\pi(a)\Delta.$$
Hence the terminology modular unitaries.
For unitaries in matrix algebras over $\A$ we use the
regular automorphism 
$\s\otimes Id_n$ to state the modular condition, where $Id_n$
is the identity automorphism of $M_n(\C)$.

{\bf Example}. For $S_\mu\in O_{nc}$ we write $P_\mu=S_\mu
S_\mu^*$. Then 
for each
$\mu,\nu$ we have a unitary
$$ u_{\mu,\nu}=\bma 1-P_\mu & S_\mu S_\nu^*\\ S_\nu S_\mu^* &
1-P_\nu\ema.$$ It is simple to check that this a self-adjoint
unitary satisfying the modular condition.

\begin{defn} Let $u_t$ be a continuous path of modular unitaries  in the 
$*$-subalgebra $\A$ such
that $u_t\s(u_t^*)$ and $u^*_t\s(u_t)$ are also {\bf continuous} paths in
$F$ (this is not guaranteed since $\s$ is not generally bounded). 
Then we say that $u_t$ is a {\bf modular homotopy}, and say that $u_0$ and
$u_1$ are {\bf modular homotopic}. If $u$ and $v$ are modular unitaries, we say 
that $u$ is {\bf equivalent} to $v$ if there exist $k,m\geq 0$ so that 
$u\oplus 1_k$ is modular homotopic to $v\oplus 1_m.$
\end{defn}

\begin{lemma} The relation defined above is an equivalence relation.
Moreover, if $u$ is a modular unitary and $k\geq 0$ then
$1_k\oplus u$ is modular homotopic to $u\oplus 1_k.$
The binary operation on equivalence classes in $U_\s$, given by
$[u]+[v]:=[u\oplus v]$ is well-defined and abelian.
\end{lemma}

\begin{proof}  It is straightforward to show that this is an
equivalence relation. To see that $1_k\oplus u$ is modular homotopic to
$u\oplus 1_k$ it suffices to do this for $k=1.$ If $u\in M_m(\A)$ is a 
modular unitary then
let $x_0\in M_{m+1}(\C)$ be the (backward) shift matrix whose action
on the standard basis of $\C^{m+1}$ is given by 
$x_0(\overline{e}_k)=\overline{e}_{k-1} (mod)(m+1).$ Then,
$x_0(1\oplus u)x_0^*=(u\oplus 1).$ Let $\{x_t\}$ be a coninuous path of 
scalar unitaries from $x_0$ to $x_1=1_{m+1}.$ Of course each $x_t\in M_{m+1}(F)$
as well. Since $\s(x_t)=x_t,$ one easily checks that 
$\{x_t(1\oplus u)x_t^*\}$ is a modular homotopy from $u\oplus 1$ to $1\oplus u.$

To see that addition is well-defined, we must show that $u\oplus v$ is 
equivalent to $(u\oplus 1_k)\oplus (v\oplus 1_m).$ But this equals 
$u\oplus (1_k\oplus v)\oplus 1_m.$ By the previous argument this is
equivalent to $u\oplus (v\oplus 1_k)\oplus 1_m$ which equals
$(u\oplus v)\oplus 1_{k+m}$ which is equivalent to $(u\oplus v).$

To see that addition of classes is abelian let $u,v$ be modular unitaries.
By adding on copies of the identity, we can assume that $u$ and $v$ are both
the same size matrices. Hence, it suffices to show 
that $u\oplus v$
is modular homotopic to $v\oplus u.$  To this end, let
$$ R_t=\bma \cos(t) & \sin(t)\\ -\sin(t) & \cos(t)\ema$$
for $t\in [0,\pi/2]$. Let $w_t=R_t(u\oplus v)R_t^*$. Then we have
$$ w_t=\bma \cos^2(t)u+\sin^2(t)v&\cos(t)\sin(t)(v-u)\\
\cos(t)\sin(t)(v-u)&\cos^2(t)v+\sin^2(t)u\ema.$$ Observe that at
$t=0$ we have $u\oplus v$ and at $t=\pi/2$ we have $v\oplus u$. We
need to show that $w_t\s( w_t^*)\in M_2(F)$ for all
$t\in[0,\pi/2]$. Write $\hat u$ for $\s( u^*)$ and
similarly for $v$. Then we compute \bean &&w_t\s(
w_t^*)\nno&=&\bma \cos^2(t)u+\sin^2(t)v&\cos(t)\sin(t)(v-u)\\
\cos(t)\sin(t)(v-u)&\cos^2(t)v+\sin^2(t)u\ema\bma \cos^2(t)\hat
u+\sin^2(t)\hat v
&\cos(t)\sin(t)(\hat v-\hat u)\\
\cos(t)\sin(t)(\hat v-\hat u)&\cos^2(t)\hat v+\sin^2(t)\hat
u\ema\nno
&=& \bma \cos^2(t)u\hat u+\sin^2(t)v\hat v & \cos(t)\sin(t)(v\hat
v-u\hat u)\\ \cos(t)\sin(t)(v\hat v-u\hat u) &\cos^2(t)v\hat
v+\sin^2(t)u\hat u\ema\eean and since both $u\hat u$ and $v\hat v$
lie in $F$, this is in $M_2(F)$. The other half of the modular
condition follows by replacing $u,v$ by $u^*,v^*$.
\end{proof}

We can now also see why the usual proof that the inverse of $u$ is
$u^*$ in $K_1(A)$ is not available to us. This usual proof is as
follows. In the $K_1$ setting one uses: 
$u\oplus v=(u\oplus 1)(1\oplus v)\sim
(1\oplus u)(1\oplus v)=(1\oplus uv)$, so that addition in
$K_1$ arises from multiplication of unitaries, and hence
$[u]+[u^*]=[uu^*]=[1]=0$. However, in the modular setting, while the 
homotopy from $u\oplus 1$ to $1\oplus u$ is a modular homotopy 
in $U_\s$ by the last Lemma, the homotopy from $(u\oplus
1)(1\oplus v)$ to $(1\oplus u)(1\oplus v)$ is not in general. The
multiplication on the right by $(1\oplus v)$ breaks the
modular condition. In particular, the product of two modular unitaries
{\bf need not be a modular unitary}. 

\begin{lemma} If $u\in M_k(F)$ is unitary then $u\oplus u^*\sim
1$.
\end{lemma}

\begin{proof} There is a path $w_t$ from $u\oplus u^*$ to $1$ through
unitaries in $M_k(F)$ and so $w_t\s(
w_t^*)=1$ for all $t$ and hence we find $u\oplus u^*\sim
1$.
\end{proof}

We now formalise the above discussion. Compare the following
 with \cite[Definition 4.8.1]{HR}

\begin{defn} Let $K_1(\A,\s)$ be the abelian semigroup of equivalence
classes of modular unitaries $u$ over $\A$ under the equivalence relation
$u$ is {\bf equivalent} to $v$ if there exist $k,m\geq 0$ so that $u\oplus 1_k$
is modular homotopic to $v\oplus 1_m.$ The following relations hold in
$K_1(\A,\s)$
\bean 1)&& [1]=0,\nno 2)&&
[u]+[v]=[u\oplus v],\nno 3)&& \mbox{If }u_t,\ t\in[0,1]\ \mbox{is a
continuous paths of unitaries in }M_k(\A)\mbox{ with}\ u_t\s(u_t^*) \mbox{ and}\ 
u_t^*\s(u_t)\nno && \mbox {continuous over}\ F\ \mbox{then}\ [u_0]=[u_1].\eean
\end{defn}

\begin{cor} If $u\in M_k(F)$ then $-[u]=[u^*]$ in $K_1(A,\s)$.
\end{cor}
We can make $K_1(A,\s)$ a group by the Grothendieck construction,
but this is not needed here.
The following lemma is a clear departure from the situation in
 \cite{pr} (it implies that the `obvious' map 
from $K_0(M(F,A))$
to $K_1(A,\s)$ is not well-defined).

\begin{lemma}\label{noinv} Recall, for all paths $\mu,\nu$ with 
$P_\mu=S_\mu S_\mu^*$ we have a modular unitary
$$ u_{\mu,\nu}=\bma 1-P_\mu & S_\mu S_\nu^*\\ S_\nu S_\mu^* &
1-P_\nu\ema.$$ 
 Then there is a modular homotopy
$ u_{\mu,\nu}\sim u_{\nu,\mu}.$
\end{lemma}

\begin{proof} We do the homotopy in two steps. The first is given
by conjugating $u_{\mu,\nu}$ by the scalar unitary matrix
$$ \bma \cos\theta&\sin\theta\\-\sin\theta&\cos\theta\ema,\ \
\theta\in[0,\pi/2],$$ which takes us to
$$ \bma 1-P_\nu & -S_\nu S_\mu^*\\-S_\mu S_\nu^* & 1-P_\mu\ema.$$
Then for $\theta\in[0,\pi]$ we consider
$$\bma 1-P_\nu & e^{i\theta}S_\nu S_\mu^*\\e^{-i\theta}S_\mu S_\nu^* &
1-P_\mu\ema.$$ The reader will readily confirm that these two
homotopies are modular.
\end{proof}

{\bf Example}. More generally, if $\s$ is a regular automorphism of a unital
$*$-algebra $\A$ with fixed point algebra $F$, $v\in \A$ is a partial
isometry with range and source projections in $F$, and furthermore {\bf if}
$v\s(v^*),\ v^*\s(v)$ lie in $F$, then 
$$u_v=\bma 1-v^*v & v^*\\ v & 1-vv^*\ema$$
is a modular unitary over $\A$, as the reader may check. The proof
of Lemma \ref{noinv} applies to these unitaries to show that $u_v\sim
u_{v^*}$.
 

\begin{lemma}\label{centre} Let $(\A,\HH,\D)$ be 
our modular spectral triple relative
to $(\cn,\tau_\Delta)$ 
and set $F=\A^\s$ and $\s:\A\to\A$.
Let $L^\infty(\Delta)=L^\infty(\D)$ be the von Neumann algebra generated by 
the spectral projections of $\Delta$
then $L^\infty(\Delta)\subset{\mathcal Z}(\cM)$.
Let $u\in\A$ be a unitary, then $\pi(u)Q\pi(u^*)\in\cM$ and 
$\pi(u^*)Q\pi(u)\in\cM$ for all spectral 
projections $Q$ 
of $\D$, if and only if $u$ is modular. That is, $\pi(u)\Delta \pi(u^*)$ 
and $\pi(u^*)\Delta \pi(u)$
(or $\pi(u)\D \pi(u^*)$ and $\pi(u^*)\D \pi(u)$) 
are both affiliated to $\cM$ if and only if $u$ is modular.
\end{lemma}

\begin{proof} First, $L^\infty(\Delta)$ is an abelian algebra. 
By Lemma \ref{commutes} all the $\Phi_k$ are in $\cm$ and since
the $\Phi_k$ are also the spectral projections of $\Delta$, we have
$L^\infty(\Delta)$ is contained in the
centre. (Note that this extends the fact
that $\D$ commutes with $\pi(F)=\pi(\A^\s)$).
Next we observe that $\pi(u)Q\pi(u^*)$ is a projection in $\cn$. For one
direction, suppose $u$ is modular, then we
have
\bean \Delta^{-1}\pi(u)Q\pi(u^*)\Delta&=&
\Delta^{-1}\pi(u)\Delta\Delta^{-1}Q\Delta\Delta^{-1}\pi(u^*)\Delta\ \ \ \ \ \ \ Q\in\cM\nno
&=&\pi(\s(u))Q\pi(\s(u^*))\nno
&=& \pi(u)\pi(u^*\s(u))Q\pi(\s(u^*))\nno
&=&\pi(u)Q\pi(u^*\s(u)\s(u^*)),\ \ \ \ \ \ \ u^*\s(u)\in F\nno
&=&\pi(u)Q\pi(u^*).\eean
Hence $\pi(u)Q\pi(u^*)$ commutes with $\Delta$, and so is in $\cM$. 
Similarly, $u\s(u^*)\in F$ implies
that $\pi(u^*)Q\pi(u)\in\cM$. On the other hand if $\pi(u)Q\pi(u^*)\in\cM$
then
$$ \pi(u)Q\pi(u^*)=\Delta^{-1}\pi(u)Q\pi(u^*)\Delta=\pi(\s(u))Q\pi(\s(u^*))$$
and so we have 
$$Q=\pi(u^*\s(u))Q\pi(\s(u^*)u)=Q+[\pi(u^*\s(u)),Q]\pi(\s(u^*)u).$$
As $\s(u^*)u$ is  invertible, we see that
$[\pi(u^*\s(u)),Q]=0$. Since $\pi(u^*\s(u))\in\pi(\A)$, and commutes with all 
$Q$, we have $\pi(u^*\s(u))\in\cm$ and so
lies in $\pi(F)=\cM\cap\pi(\A)$. That is, $u^*\s(u)\in F.$ 
Similarly, $\pi(u^*)Q\pi(u)\in\cM$ implies that $u\s(u^*)\in F.$
\end{proof}

The fundamental aspect of the last lemma is that modular
unitaries conjugate $\Delta$ to an operator affiliated to $\cM$, and so 
$u\Delta u^*$ commutes with $\Delta$ (and $u\D u^*$ commutes with $\D$). 
We will next show  that there is a pairing between (part of) modular
$K_1$ and modular spectral triples. To do this, we are going to use
the analytic formulae for spectral flow 
in \cite{CP2}. 

\section{An $\LL^{(1,\infty)}$ local index formula}

In this Section we will couch our results in terms of the notion of a 
modular spectral triple. That is we will assume properties (0) to (3)
listed in Section 3 apply. Of course at this time the only examples 
we have presented are the matrix algebras over the smooth subalgebra $O_{nc}$
of the Cuntz algebra. However, we know from work in progress that there are 
other examples and hence it is worth arguing directly from
the general properties and avoiding the explicit formulae of the Cuntz example.

\subsection{The spectral flow formula: correction
terms} The spectral flow formula of \cite{CP2} is, a priori,
complicated in our setting. This is because we are computing
the spectral flow between two operators which {\bf are not}
unitarily equivalent via a unitary {\bf in} $\cM$. Thus we must consider
$\eta$-type correction terms.
We will also recognise that the spectral flow we are calculating
depends on the choice of trace $\phi$ on $\cM$ and use the notation
$sf_\phi$.
We now quote  \cite[Corollary 8.11]{CP2}.

\begin{prop} Let $(\A,\HH,\D_0)$ be an odd unbounded 
$\theta$-summable semifinite
spectral triple relative to $(\cM,\phi)$. For any $\epsilon>0$ we
define a one-form $\al^\epsilon$ on $\cM_0=\D_0+\cM_{sa}$ by
$$\al^\epsilon(A)=\sqrt{\frac{\epsilon}{\pi}}\phi(Ae^{-\epsilon\D^2})$$
for $\D\in\cM_0$ and $A\in T_\D(\cM_0)=\cM_{sa}$. Then the
integral of $\al^\epsilon$ is independent of the piecewise $C^1$
path in $\cM_0$ and if $\{\D_t\}_{t\in[a,b]}$ is any piecewise
$C^1$ path in $\cM_0$ then $$
sf_\phi(D_a,D_b)=\sqrt{\frac{\epsilon}{\pi}}\int_a^b
\phi(\D_t'e^{-\epsilon\D_t^2})dt
+\frac{1}{2}\eta_\epsilon(\D_b)
-\frac{1}{2}\eta_\epsilon(\D_a)+\frac{1}{2}
\phi\left([\ker(\D_b)]-[\ker(\D_a)]\right).$$
where the following integral converges for all $\epsilon>0$
$$\eta_\epsilon(\D)=\frac{1}{\sqrt{\pi}}\int_\epsilon^\infty
\phi(\D e^{-t\D^2})t^{-1/2}dt.$$

\end{prop}
We note that the 
 $\eta$ terms are measures of  $\phi$-spectral asymmetry.
We
will show that for the pair $\D,\tau_\Delta$ 
we use on the Cuntz algebra, and the kinds of
perturbations we consider, these $\eta$ terms vanish. Moreover we
will show that the $\tau_\Delta$-dimension of the kernel of $\D$ is
unchanged by the particular type
 of perturbations we consider, so these correction
terms will cancel. First we must show that we are actually working
with the right kinds of perturbations, that is, elements
in $\cm_{sa}.$

\noindent{\bf Notation}. We denote the densely defined spatial homomorphism
on $\cn$, $T\mapsto \Delta^{-1}T\Delta$ by $\s_i(T)$, so that for
$a\in\A$ we have $\pi(\s(a))=\s_i(\pi(a)).$ We observe that $\cm$,
and $\pi(\A)$ are in the domain of $\s_i$, and that $\cm$ is exactly the
fixed point subalgebra of $\s_i.$

\begin{lemma} Let $(\A,\HH,\D)$ be a modular spectral triple. If $u$ is a 
modular unitary, then $\pi(u)[\D,\pi(u^*)]\in\cM_{sa}$. This is a {\bf key fact} 
which allows us to directly use results
about semifinite spectral flow in $(\cm,\tau_\Delta)$ from \cite{CP2}.
\end{lemma}

\begin{proof} We just compute the action of $\s_i$ on $\pi(u)[\D,\pi(u^*)].$
As observed above the operator $[\D,\pi(u^*)]\in \pi(\A),$ and we
easily calculate:
$$\s_i(\pi(u)[\D,\pi(u^*)])=\pi(\s(u))[\D,\pi(\s(u^*))]=
\pi(uu^*\s(u))[\D,\pi(\s(u^*))]=\pi(u)[\D,\pi(u^*)].$$
\end{proof}

\begin{rems*}
In the following few pages we will sometimes abuse notation and write $a$ 
in place of $\pi(a)$ for $a\in\A$ in order to make our formulae more readable. 
Whenever we do this, however, we will use $\s_i(\cdot)=\Delta^{-1}(\cdot)\Delta$
the spatial version of the algebra homomorphism, $\s$. We will generally use
the spatial version $\s_i$ when in the presence of operators not in $\pi(\A).$
\end{rems*}

\begin{lemma}\label{corrections} Let $(\A,\HH,\D)$ be our 
modular spectral triple for the Cuntz algebra
and let
$u$ be a modular  unitary. Then
$$\tau_\Delta([\ker(\D)]-[\ker (u\D u^*)])=
\tau_\Delta((1-\s_i(u^*)u)[\ker(\D)])=\tau(1-\s(u^*)u),$$
and for all $\epsilon>0$,
$ \eta_\epsilon(u\D u^*)
=\tau(\s(u^*)u)\eta_\epsilon(\D).$
\end{lemma}

\begin{proof} We show the second equality first. 
By the $\s_i$-invariance of $\tau_\Delta$ and the fact that $\s_i(u^*)u\in\cm$
we have, using Lemma \ref{tracesplit} in the last equality:
$$ \tau_\Delta(u\D u^*e^{-t(u\D u^*)^2})=\tau_\Delta(u\D e^{-t\D^2}u^*)=
\tau_\Delta(\s_i(u)\D e^{-t\D^2}\s_i(u^*))
=\tilde\tau(u\Delta \D e^{-t\D^2}\s_i(u^*))$$
$$=\tilde\tau(\Delta\D e^{-t\D^2}\s_i(u^*)u)=
\tilde\tau(\Delta\s_i(u^*)u\D e^{-t\D^2})=
\tau_\Delta(\s_i(u^*)u\D e^{-t\D^2})=
\tau(\s(u^*)u)\tau_\Delta(\D e^{-t\D^2}).$$
                                      
Thus, we have 
$$\eta_\epsilon(u\D u^*)=\frac{1}{\sqrt{\pi}}
\int_\epsilon^\infty\tau(\s(u^*)u)\tau_\Delta(\D e^{-t\D^2} )t^{-1/2}dt
=\tau(\s(u^*)u)\eta_\epsilon(\D),$$
as was to be shown. For the kernel we simply observe that
$[\ker (u\D u^*)]=u[\ker(\D)]u^*\in\cn,$
so that 
$$\tau_\Delta([\ker (u\D u^*)])=\tilde\tau(\Delta u[\ker(\D)]u^*)=
\tilde\tau(u[\ker(\D)]\Delta\Delta^{-1}u^*\Delta)=
\tau_\Delta(\s_i(u^*)u[\ker(\D)]).$$
Then, by Lemma \ref{tracesplit}, 
$\tau_\Delta(\s_i(u^*)u[\ker(\D)])=\tau(\s(u^*)u)\tau_\Delta([\ker(\D)])
=\tau(\s(u^*)u)\cdot 1.$
\end{proof}




{\bf If} we have a modular unitary for which we 
have both $\eta_\epsilon(u_v\D u_v)-\eta_\epsilon(\D)=0$ and
$\phi([\ker\D])-\phi([\ker u\D u^*])=0$,
we may apply the Laplace transform technique discussed in
\cite[Section 9]{CP2} to reduce the $\theta$-summable formula to
the finitely summable formula. For $r>0$ this gives us
\begin{equation}\label{sf-formula} sf_\phi(\D,u\D
u^*)=
\frac{1}{C_{1/2+r}}\int_0^1\phi(u[\D,u^*](1+(\D+tu[\D,u^*])^2)^{-1/2-r})dt.
\end{equation}
We are now in a position to apply the methods employed in the 
proof of the semifinite local index
formula, \cite{CPS2} or \cite{CPRS2}, to compute an index pairing.

\subsection{A local index formula for the Cuntz algebras}

\begin{lemma}[cf \cite{CP2}]\label{oneform} 
Let $(\A,\HH,\D)$ be our $(1,\infty)$-summable
  modular spectral for the Cuntz algebra
triple for a matrix algebra $\A$ over $O_{nc}$. Let $\cM=\cn^\s$ be the fixed point
algebra for the modular automorphism group. The functional $\alpha$ defined on
the self adjoint elements   $\cM_{sa}$ of $\cM$ by
$$\alpha_{S}(T)=\widehat\tau(T(1+(\D+S)^2)^{-s/2}),\ \ \ T\in \cM_{sa}$$ 
for $s>1$ is an exact one form
on the tangent space to the
affine space $\cm_0=\cm_{sa}+\D$ of $\cm_{sa}$ perturbations of $\D$.
\end{lemma}

This fact is all that we need to calculate $\widehat\tau$-spectral flow along
paths in the affine space  $\cm_0$.
We will be  interested in $\widehat\tau$-spectral flow along the linear 
path joining
$\D$ to $\D +u[\D,u^*]$ where $u$ is a unitary in
$End_{F}(X)$ such that $u[\D,u^*]\in \cM_{sa}$. Since modular unitaries, $u$
satisfy these requirements,
we can now produce a formula for spectral flow which is analogous to the local 
index formula in noncommutative geometry. We remind the reader that 
$\tau_\Delta=\widehat\tau\circ\Psi$ where $\Psi:\cn\to\cm$ is the canonical
expectation, so that $\tau_\Delta$ {\bf restricted to} $\cm$ is $\widehat\tau.$

\begin{thm}\label{first} Let
$(\A,\HH,\D)$ be the 
$(1,\infty)$-summable, 
modular spectral triple for the Cuntz algebra
we have constructed previously.
 Then for any modular
unitary 
such that the difference of
eta terms $\eta_\epsilon(u\D u^*)-\eta_\epsilon(\D)$ and
$\widehat\tau([\ker\D])-\widehat\tau([\ker u\D u^*])$ 
vanishes, and for any
Dixmier trace $\widehat{\tau}_{\tilde\omega}$ associated to $\widehat\tau$, 
we have
spectral flow as an actual limit
$$sf_{\widehat\tau}(\D,u\D u^*)=\frac{1}{2}\lim_{s\to 1+}
(s-1)\widehat\tau(u[\D,u^*](1+\D^2)^{-s/2})
=\frac{1}{2}\widehat{\tau}_{\tilde\omega}(u[\D,u^*](1+\D^2)^{-1/2})
=\tau\circ\Phi(u[\D,u^*]).$$
The functional on $\A\otimes\A$ defined by
$a_0\otimes a_1\mapsto
\frac{1}{2}\lim_{s\to 1^+}(s-1)\tau_\Delta(a_0[\D,a_1](1+\D^2)^{-s/2})$
is a $\s$-twisted $b,B$-cocycle (see the proof below for the definition). 
\end{thm}

\begin{proof} First we observe that by \cite[Lemma 6.1]{CPS2}, the difference
$$(1+(\D+t u[\D,u^*])^2)^{-s/2}-(1+\D^2)^{-s/2}$$
is uniformly bounded in trace class norm for $t\in [0,1]$ and $s\in (1,4/3)$. 
Hence in the spectral flow formula (\ref{sf-formula}), by the simple change of 
variable $r=1/2(s-1),$ we may write
$$C_{s/2}\;sf_{\widehat\tau}(\D,u\D u^*)=
\widehat\tau(u[\D,u^*](1+\D^2)^{-s/2})
+\mbox{remainder}.$$
Where the remainder is bounded as $s\to 1^+.$ Multiplying this equation by 
$(s-1)/2$
and taking the limit as $s\to 1^+$ recalling that 
$C_{s/2}=\frac{\Gamma((s-1)/2)\Gamma(1/2)}{\Gamma(s/2)}$ 
so that $(s-1)/2\;C_{s/2}\to 1,$ we get:
$$sf_{\widehat\tau}(\D,u\D u^*)=
\frac{1}{2}\lim_{s\to 1^+}(s-1)\widehat\tau(u[\D,u^*](1+\D^2)^{-s/2}).$$
Now by the proof of Lemma \ref{phiclosed}, $u[\D,u^*]$ is in $\A=O_{nc}$
and since it is also in $\cm$ it is in $F_c$
and so by Proposition \ref{funnystate} this last limit equals
$\tau\circ\Phi(u[\D,u^*])$ as claimed.

To see that we obtain a $\s$-twisted cocycle, we denote by $\theta$ the
functional
$$\theta(a_0,a_1)=
\frac{1}{2}\lim_{s\to 1^+}(s-1)\tau_\Delta(a_0[\D,a_1](1+\D^2)^{-s/2}),$$
and observe that by the proof of Lemma \ref{phiclosed} the elements
$[\D,a_1]$ are in $\A$ and so by Proposition \ref{funnystate} we see that
not only do these limits exist, but in fact,
$$\theta(a_0,a_1)=\tau\circ\Phi(a_0[\D,a_1]).$$
By definition  $B^\s\theta(a_0)=\theta(1,a_0),$ and so by Lemma \ref{phiclosed}
$$(B^\s\theta)(a_0)=\lim_{s\to 1+}(s-1)\tau_\Delta([\D,a_0](1+\D^2)^{-1/2-r})=0$$ 
By definition, 
\bean b^\s\theta(a_0,a_1,a_2)&=&\theta(a_0a_1,a_2)-\theta(a_0,a_1a_2)
+\theta(\s(a_2)a_0,a_1)\\
&=& -\tau\circ\Phi(a_0[\D,a_1]a_2) + \tau\circ\Phi(\s(a_2)a_0[\D,a_1])\\ 
\eean
This is $0$ by the KMS condition (see the Remark prior to Corollary 3.4)
for the state $\psi=\tau\circ\Phi.$
Thus, both $b^\s\theta=0$ and $B^\s\theta=0,$ and we're done.
\end{proof}
{\bf Remark}. Spectral flow in this setting is independent of the path 
joining the endpoints of unbounded self adjoint operators affiliated to 
$\mathcal M$ however
it is not obvious that this is enough to show that it is constant on homotopy 
classes of modular unitaries. This latter fact is true but the 
proof is lengthy and we defer it until we have a fuller understanding of
the structure of the modular unitaries.

\begin{thm} We let $(O_{nc}\otimes M_2,\HH\otimes\C^2,\D\otimes 1_2)$ 
be the modular spectral triple of $(O_{nc}\otimes M_2)$
and $u$ a modular unitary of the form
$$u_{\mu,\nu}=\bma 1-P_\mu & S_\mu S_\nu^*\\ S_\nu
S_\mu^* & 1-P_\nu\ema.$$ Then the spectral flow is positive being given by
\bean sf_{\tau_\Delta}(\D, u\D u^*)&=&
(|\mu|-|\nu|)(n^{-|\nu|}-n^{-|\mu|})
\in (n-1)\Z[1/n]\eean
\end{thm}

\begin{proof} Once we have verified that the difference of eta
 terms and the difference of kernel corrections 
vanish, this is just a computation. In fact, by Lemma \ref{corrections},
$$ \eta_\epsilon(u\D u^*)=\tau(\s(u^*)u)\eta_\epsilon(\D)=
\tau(\s(u^*)u)\int_\epsilon^\infty\left(\sum_{k\in\Z}ke^{-tk^2}\right)dt=0=
\int_\epsilon^\infty\left(\sum_{k\in\Z}ke^{-tk^2}\right)dt=\eta_\epsilon(\D).$$

For the kernel corrections we use Lemma \ref{corrections} and first compute  $1-\s(u_v^*)u_v$,
noting that $$\s(v)(1-v^*v)=\s(v)\s(1-v^*v)=\s(v-vv^*v)=0.$$
$$1-\s(u_v^*)u_v=1-\s(u_v)u_v=\bma v^*v-\s(v^*)v &0\\ 0 & vv^*-\s(v)v^*\ema.$$
For $\tau(1-\s(u_v^*)u_v)$
we use the KMS property of $\psi=\tau\circ\Phi$:
\bean
\tau(1-\s(u_v^*)u_v)=\tau(v^*v-\s(v^*)v)+\tau(vv^*-\s(v)v^*)
=\tau(v^*v-vv^*)+\tau(vv^*-v^*v)=0.\eean
Hence both the eta terms and kernel corrections vanish, and the
spectral 
flow can be computed
from the integral of the exact one form of Lemma \ref{oneform}.

For the computation we use a calculation in the proof of Lemma \ref{phiclosed}
to get
\bean &&u_{\mu,\nu}[\D\otimes 1_2 ,u_{\mu,\nu}]
=\bma 1-P_\mu & S_\mu S_\nu^*\\ S_\nu
S_\mu^* & 1-P_\nu\ema\bma 0 & {[\D ,S_\mu S_\nu^*]}\\
{[\D , S_\nu S_{\mu}^*]} & 0\ema\\
&=&\bma 1-P_\mu & S_\mu S_\nu^*\\ S_\nu
S_\mu^* & 1-P_\nu\ema\bma 0 & (|\mu|-|\nu|)S_\mu S_\nu^*\\
(|\nu|-|\mu|)S_\nu S_\mu^* & 0\ema=(|\mu|-|\nu|)\bma -P_\mu & 0\\ 0 &
P_\nu\ema.\eean
So using Theorem \ref{first} and our previous computation of the
Dixmier trace, Proposition \ref{dixycomp}, we have
$$sf_{\tau_\Delta}(\D,u_{\mu,\nu}\D u_{\mu,\nu})=(|\mu|-|\nu|)\tau(P_\nu-P_\mu)
=
(|\mu|-|\nu|)(n^{-|\nu|}-n^{-|\mu|}).$$
This number is always positive as the reader may check, and is
contained in $(n-1)\Z[1/n]$, the integer polynomials in $1/n$ all of
whose coeffficients have a factor of $(n-1)$.  
\end{proof}
\begin{rems*}
We observe that since this unitary $u_{\mu,\nu}$ is self-adjoint the spectral
flow {\bf cannot} be interpreted simply as the index of the Toeplitz compression
of $u_{\mu,\nu}$ by the non-negative spectral projection of
$\D\otimes 1_2$: for one thing this ``Toeplitz compression'' is not in $\cm$
and if it were in $\cm$ its index would have to be $0$.
Next we use the viewpoint provided by the noncommutative APS theory of 
\cite{CPR}. This gives a partial explanation of the numerical values of 
the spectral flow obtained for the Cuntz algebras.
\end{rems*}
\begin{cor}  Let $(O_{nc}\otimes M_2,\HH\otimes\C^2,\D\otimes 1_2)$ 
be the modular spectral triple of the theorem
 and $u$ a modular unitary of the form
$u_{v}$, where $v=S_\mu S_\nu^*$ so that $v^*v=S_\nu S_\nu^*$ and
$vv^*=S_\mu S_\mu^*$ are both in $F.$ Let $(X,\D\otimes 1_2)$ be the 
Kasparov module 
for $O_n\otimes M_2,F\otimes M_2$
described earlier. Then from the pairing
$$ K_0(M(F\otimes M_2,O_n\otimes M_2))\times (X,\D\otimes 1_2)\to K_0(F)$$
we have the classes of the projections 
$${\rm Index}(PvPv^*:vPv^*(X)\to vv^*P(X))\ \ {\rm and}\ \ {\rm
  Index}(Pv^*Pv:v^*Pv(X)\to v^*vP(X)) \in K_0(F).$$
These two classes are negatives of each other in $K_0(F)$, but
$$sf_{\tau_\Delta}(\D,u_{v}\D u_{v})=
\tau_\Delta({\rm Index}(PvPv^*:vPv^*X\to vv^*PX)) + 
\tau_\Delta({\rm Index}(Pv^*Pv:v^*PvX\to v^*vPX)),$$
where here we apply $\tau_\Delta$ to the difference of projections defining 
the index as a
difference of $F$-modules.
\end{cor}

\begin{proof}
In \cite{CPR} Lemma 3.5 and Theorem 5.1, the authors used the following
operators and indices:
$${\rm Index}(PvP:v^*vP(X)\to vv^*P(X))\;\;\;{\rm and}\;\;\;
{\rm Index}(Pv^*P:vv^*P(X)\to v^*vPX).$$
Each index is the {\bf exact} negative of the other, and so 
if we evaluate both with $\tau_\Delta$ and add we get exactly $0.$
The $K_0(F)$ elements given by the indices of these two operators are the 
same as the ones considered in this Corollary. However, the point of view 
of this Corollary is to consider mappings from say 
the non-negative spectral subspace of $v\D v^*$ (i.e., $vPv^*(X)$) to 
the non-negative spectral subspace of $vv^*\D$ (i.e., $vv^*P(X)$). 
Here we get quite a different answer.

Let $v=S_\mu S_\nu^*$, and $m=|\mu|-|\nu|$, $m>0$; so that $vv^*=S_\mu S_\mu^*.$
A simple computation on monomials $S_\alpha S_{\beta}^*$ gives us the 
{\bf key fact} that : $v\Phi_k v^* = vv^*\Phi_{k+m}$ for all $k\in\Z.$
This easily implies that $vPv^*=vv^*(\sum_{k\geq m}\Phi_k)\leq vv^*P$
so that $(PvPv^*)vPv^*=vPv^*$ and so $\ker(PvPv^*)=\{0\}.$ This also shows that:
$${\rm cokernel}(PvPv^*:vPv^*(X)\to vv^*P(X))=
vv^*P(X)\ominus vv^*(\sum_{k\geq m}\Phi_k)(X)=\sum_{k=0}^{m-1}vv^*\Phi_k(X).$$
Similarly, $v^*Pv=\sum_{k\geq -m}v^*v\Phi_k \geq v^*vP$ so that
$(Pv^*Pv)v^*Pv=v^*vP$ and so $Pv^*Pv$ is {\bf onto} $v^*vP(X).$ That is,
${\rm cokernel}(Pv^*Pv)=\{0\}.$ This also shows that
$${\rm kernel}(Pv^*Pv:v^*Pv(X)\to v^*vP(X))=
\sum_{k\geq -m}v^*v\Phi_k(X) \ominus v^*vP(X)=\sum_{k=-m}^{-1}v^*v\Phi_k(X).$$
 
To see that these indices are negatives in $K_0(F)$ it suffices to see
the equivalence between the two projections 
$\sum_{k=0}^{m-1}vv^*\Phi_k$ and $\sum_{k=-m}^{-1}v^*v\Phi_k.$  This
is obtained from our {\bf key fact} above:
$$ (v\Phi_k)(\Phi_kv^*)=vv^*\Phi_{k+m},\ \ \ (\Phi_kv^*)(v\Phi_k)=v^*v\Phi_k.$$
This is of course the Murray-von Neumann equivalence which
$\tau_\Delta$ does not respect.

Assume then that $m>0$. Applying $\tau_\Delta$
we have
$$ \tau_\Delta(\mbox{Index}(PvPv^*))=-m\tau(vv^*)=-\frac{m}{n^{|\mu|}},$$
while 
$$\tau_\Delta(\mbox{Index}(Pv^*Pv))=m\tau(v^*v)=\frac{m}{n^{|\nu|}}.$$
The case $m<0$ is similar.
\end{proof}

{\bf Remark}. This Corollary makes it clear that our new
index pairings are non-trivial precisely because $\tau_\Delta$ does
not induce a
map on $K_0(End^0_F(X))$. Of course $\widehat\tau_{\omega}$ just becomes the trace on
elements of $F$, but $\tau_\Delta$ is a weight on $\cn$ and so on
$End_F^0(X)$, which is Morita equivalent to $F$. However, since $\tau_\Delta$ 
is not a trace on $\cn$ it does
not respect all Murray-von Neumann equivalences in $\cn$, and
so does not give a well-defined map on $K$-theory. 
So we may think of the spectral flow invariant associated to
$u_{\mu,\nu}$ as a measure of the failure of $\tau_\Delta$ to respect the
Murray-von Neumann equivalence between $\mbox{Index}(PvPv^*)$
and $-\mbox{Index}(Pv^*Pv)$.

More generally we have
\bean&& sf_{\tau_\Delta}(\D_2,u_v\D_2 u_v)\nno
&=&\mbox{Index}_{\tau_\Delta}(P_2u_vP_2u_v)=
\mbox{Index}_{\tau_\Delta}\left(\bma (1-vv^*)P+PvPv^* & 0\\ 0 &
  (1-v^*v)P+Pv^*Pv\ema\right).\eean
Since $P(1-vv^*)=(1-vv^*)P$ is an isomorphism from $(1-vv^*)PX$ to
itself, and similarly for $(1-v^*v)P$, we see that the index is
precisely the sum of the indices of $PvPv^*$ from $vPv^*X$ to
$vv^*PX$, 
and $Pv^*Pv$
from $v^*PvX$ to $v^*vPX$. 
Hence the spectral flow for modular unitaries of the form $u_v$ arises
precisely because $\tau_\Delta$ does not induce a homomorphism on
$K_0(End_F^0(X))$. 

Our arguments here rely on the vanishing of the difference of eta
terms and kernel corrections. 
In the general case these eta and kernel terms contribute
and may have cohomological significance. We will return to this more general
set up in a future work.

\section{Concluding Remarks}
\subsection{Relative entropy}
In this subsection we give a physical interpretation of our index.
Let $u$ be a modular unitary over $O_n$. Recall that $\psi$ is the
state on $O_n$ defined by $\psi=\tau\circ\Phi.$ Let $\psi_u$ be the state
$\psi\circ Ad u$ on $O_n$ defined by
$\psi_u(a) =\psi(u^*au), a\in A$. The modular group for
$\psi_u$ is $t \to u\Delta^{it}u^*$
$t\in \R$. The relative entropy of a pair of KMS states
on a von Neumann algebra
was introduced by Araki \cite{Ar} (it uses explicitly 
a cyclic and separating vector).
The Hilbert space 
${\mathcal H}=\LL^2(O_n,\psi)$ has a cyclic and separating vector for the
action of $O_n$. In fact this vector remains cyclic and separating for the  
weak closure $\pi(O_n)^{\prime\prime}$ in 
$\mathcal N$ of $\pi(O_n)$. It may be thought of as the identity element in
$O_n$ but we will use the notation $\Omega$ because of the potential 
for confusion. 

For $a\in O_n$, $\psi(a) =\langle\Omega, \pi(a)\Omega\rangle$
so that we may write 
$
\psi(T)=\langle\Omega, T\Omega\rangle
$
for all $T\in \pi(O_n)^{\prime\prime}$. So we can regard $\psi_u$ and 
$\psi$ as a pair of KMS states on $\pi(O_n)^{\prime\prime}$.
Then 
 the relative entropy of $\psi_u$ and $\psi$
is \cite{Ar}
$$S(\psi_u,\psi)=-\langle\Omega, \log(u\Delta u^*)\Omega\rangle.$$
This can be written as
 $$S(\psi_u,\psi):= -\psi(u(\log\Delta) u^*-\log\Delta)$$
This is because $\Delta\Omega =\Omega$ implies that $(\log\Delta)(\Omega)=0.$
Now we can relate
the  relative entropy for this pair of KMS states on the weak closure
of $\pi(O_n)$ to spectral flow for 
the Cuntz algebra example when we have a modular
unitary $u$.
We just use the formula $\log\Delta=-(\log n )\D$ and then 
by Theorem \ref{first} we see that this relative entropy is just
 $$(\log n) \psi(u\D u^* -\D)=(\log n)\psi(u[\D,u^*])=(\log n)\tau\circ\Phi
(u[\D,u^*])=(\log n)sf(\D,u\D u^*).$$
That is, the relative entropy is just $\log n$ times the spectral
flow from $\D$ to $u\D u^*$. We remark that the relative entropy is always 
positive \cite{Ar}.

\subsection{Manifold structures} In \cite{PRS2} it was shown that many of 
the (tracial) examples of semifinite spectral triples constructed for graph and 
$k$-graph algebras satisfied natural generalisations of Connes' axioms for 
noncommutative manifolds, \cite{C1}.

Much of the discussion of \cite{PRS2} can be applied verbatim to the triple
$(O_{nc},\HH,\D)$ constructed here. For instance the axiom of 
finiteness is obvious, as is Morita equivalence (spin$^c$), first order 
condition, regularity (or $QC^\infty$), and irreducibility. The reality, or 
spin, condition can be proved as in \cite{PRS2}, and we have proven the 
closedness condition in Lemma \ref{phiclosed}.

The chief differences come from the summability/dimension/absolute continuity
and crucial orientability conditions.
We have a version of summability satisfied since
$(1+\D^2)^{-1/2}\in\LL^{(1,\infty)}(\cM,\widehat\tau),$
and for $a\in O_{nc}$ nonzero and positive,
$$\lim_{s\to 1^+}(s-1)\tau_\Delta(a(1+\D^2)^{-s/2})=
\lim_{s\to 1^+}(s-1)\widehat\tau(\Psi(a)(1+\D^2)^{-s/2})=2\psi(a)>0.$$
Moreover, we have a twisted Hochschild cycle satisfying the 
(twisted) orientability condition, 
and moreover it is given {\em by the same formula} as in the 
tracial case. This cocycle is
$$ c=\frac{1}{n}\sum_{j=1}^nS_j^*\otimes S_j.$$
We have two properties to check: that it is indeed a cocycle, and that it is 
represented by the identity operator on $\HH$.
Applying the twisted Hochschild boundary gives
$$b^\sigma c=\frac{1}{n}\sum_{j=1}^n (S_j^*S_j-\sigma(S_j)S_j^*)=
\frac{1}{n}\sum_{j=1}^n(1-nS_jS_j^*)=1-\sum_{j=1}^nS_jS_j^*=0.$$
This Hochschild cycle is represented on $\HH$ by
$$\pi(c)=\frac{1}{n}\sum_{j=1}^n S_j^*[\D,S_j]
=\frac{1}{n}\sum_{j=1}^nS_j^*S_j=\frac{1}{n}\sum_{j=1}^n1=1.$$
Hence $c$ has the required representation properties, and the replacement 
of the Hochschild theory with its twisted analogue has provided us with an 
orientation cycle for the `modular spectral triple' of the Cuntz algebra.
Thus Cuntz algebras may be a prototype for
 `type III noncommutative 
one dimensional manifolds'.

\subsection{Outlook}

There are many unresolved issues raised by these examples of an index
theory for the KMS state on the Cuntz algebra.
The main point is to understand
the nature of the invariant being computed by our spectral flow 
formula for the modular unitaries.
Just as semifinite spectral triples give rise to
$KK$-classes, modular spectral triples also give rise to $KK$-classes. This
follows in the same way as the semifinite case, \cite{NR}. However, the
relationship to the $KK$-index pairing is obviously very different and we
are investigating this now.
At this time we do not see a relationship to the viewpoint of Connes
and Moscovici \cite{CoM}.

\end{document}